\documentclass[11pt,a4paper]{article}
\usepackage{amsthm}
\usepackage{amsmath,amssymb}
\usepackage[utf8]{inputenc}
\usepackage[T1]{fontenc}

\newtheorem{definition}{D\'efinition}[section]
\newtheorem{lem}{Lemme}[section]
\newtheorem{theorem}{Th\'eorème}[section]
\newtheorem{prop}{Proposition}[section]

\newtheoremstyle{remark}%
{\topsep}
{\topsep}
{\upshape}
{}
{\itshape}
{.}
{ }
{\bf\thmname{#1}\thmnumber{ \textup{#2}}\thmnote{ \textnormal{(#3)}}}
\theoremstyle{remark}
\newtheorem{rem}{Remarque}[section]
\numberwithin{equation}{section}
\usepackage{algorithm}
\usepackage{algorithmic}
\usepackage{array}
\usepackage{url}
\usepackage[french,english]{babel}

\title{Le plus grand facteur premier de la fonction de Landau }
\author{M. Deléglise, J.-L. Nicolas}

\newcommand{\Pplus}{P^{+}}

\newcommand{\symgr}{\mathfrak{S}}
\newcommand{\thmin}{\theta_{\mathrm{min}}}

\newcommand{\starp}{{\!\!\!\hphantom{p}^{\star}\!p}}
\newcommand{\starP}{\!\!\!\hphantom{P}^{\star}\!P}
\newcommand{\starq}{\!\!\!\hphantom{q}^{\star}\!q}

\newcommand{\al}{\alpha}
\newcommand{\lb}{\lambda}
\newcommand{\bt}{\beta}
\newcommand{\dt}{\delta}
\newcommand{\gm}{\gamma}
\newcommand{\veps}{\varepsilon}

\newcommand{\bigo}[1]{O\!\left(#1\right)}
\newcommand{\intfd}[2]{\left(#1,\, #2\right]}
\newcommand{\iem}{^{\text{\`eme}}}
\newcommand{\qtx}[1]{\quad\text{#1}\quad}
\newcommand{\Li}{\operatorname{Li}}
\newcommand{\abs}[1]{\left\lvert #1 \right\rvert} 
\newcommand{\intfg}[2]{\left[#1,\, #2\right)}
\newcommand{\intf}[2]{\left[#1,\, #2\right]}

\newcommand{\dsm}[1]{\mbox{$\displaystyle #1 $}}
\newcommand{\set}[1]{\left\{#1\right\}}
\newcommand{\setof}[2]{\left\{#1\ ;\ #2\right\}}
\newcommand{\sg}[1]{#1^{''}}
\newcommand{\sgg}{^{''}}
\newcommand{\cp}{{\cal P}}
\newcommand{\suppress}[1]{}
\newcommand{\pfrac}[2]{\left(\frac{#1}{#2}\right)}

\newcommand{\intfrac}[2]{\left\lfloor\frac{#1}{#2}\right\rfloor}
\newcommand{\dv}{\mid}
\newcommand{\cn}{{\cal N}}
\renewcommand{\d}{\,\mathrm{d}}

\newcommand{\R}{\mathbb{R}}

\newenvironment{notation}{\smallskip\noindent {\bf Notation}~: }%
{\par\smallskip}

\def\lsim{\mathrel{\rlap{\lower5pt\hbox{\hskip1pt$\sim$}}
    \raise1.4pt\hbox{$\hspace{0.1em}<$}}} 
\def\gsim{\mathrel{\rlap{\lower4pt\hbox{\hskip1pt$\sim$}}
    \raise1pt\hbox{$>$}}} 
\begin{document}
\selectlanguage{french}
\maketitle

\selectlanguage{english}
\begin{abstract}
After Landau, let us define $g(n)$ as the maximal order of a
permutation of the symmetric group ${\symgr}_n$ on
$n$ letters. We give several estimates of the largest prime
divisor $\Pplus(g(n))$ of $g(n)$.
\end{abstract}
\selectlanguage{french}

\paragraph{Key words~:} maximal order, symmetric group, distribution of
primes, Landau's function.

\paragraph{2000 Mathematics subject classification~:} 11A25, 11N37, 11N05.

\section{Introduction}
Dans le groupe symétrique $\symgr_n$ sur $n$ lettres chacune
des $n!$ permutations a un ordre. Dans \cite{LAN}, Landau a considéré
\[
g(n)= \max_{\sigma \in \symgr_n} (\text{ordre de } \sigma)
\]
et démontré que
\begin{equation}\label{glandau}
g(n) = \max_{\ell(M) \le n} M
\end{equation}
o\`u $\ell$ est la fonction additive (c'est à dire vérifiant
$\ell(MN) = \ell(M) + \ell(N)$ lorsque $M$ et $N$ sont premiers
entre eux) définie par
\begin{equation}\label{ldef}
\ell(p^{\al}) = p^{\al}, \quad p \text{ premier, }\ \al \ge 1.
\end{equation}
On notera que $\ell(p^0) = 0 \ne p^0 = 1$, et que
\begin{equation}\label{lgn}
\ell(g(n)) \le n.
\end{equation}
La formule \eqref{glandau} entra\^ine
\begin{equation}\label{mpgn}
M > g(n) \implies \ell(M) > n
\end{equation}
ce qui montre que les valeurs de $g(n)$ sont les nombres
en lesquels la fonction $\ell$ est \og petite \fg. On a
donc un problème d'optimisation pour la fonction arithmétique
additive $\ell$ qui rejoint les problèmes
de grandes valeurs de certaines fonctions multiplicatives 
étudiées par Ramanujan 
(cf. \cite{RAMCOL,RAMNR,NICHCN}).

Les méthodes utilisées par Ramanujan dans \cite{RAMCOL}
pour étudier
les \og highly composite numbers \fg\ ont été employées dans 
(\cite{NICAA,NICTH})
 pour obtenir des propriétés de la décomposition
en facteurs premiers de $g(n)$.
Dans \cite{NICCRAS70} il est démontré que
\[
\lim_{n \to +\infty} \frac{g(n+1)}{g(n)} = 1.
\]
Dans l'article \cite{LAN} (cf. aussi \cite{MIL})
Landau démontre l'équivalence
\begin{equation}\label{equilandau}
\log g(n) \sim \sqrt{n \log n}, \qquad n \to +\infty.
\end{equation}
Cette évaluation de $\log g(n)$ a été améliorée dans
(\cite{SHAH,SZA,MAS,MNRAA,MNRMC})
à la fois sous forme asymptotique et sous
forme effective. Voir les formules figurant
ci--dessous au paragraphe \ref{rappelsgn}.

Il existe un algorithme simple permettant de calculer la table
des valeurs de $g(n)$ pour $1 \le n \le n_0$ (cf. \cite{NICRIRO,DNZ}).

Dans \cite{DNZ}, un algorithme beaucoup plus sophistiqué permet
de calculer la décomposition en facteurs premiers
de $g(n)$ pour $n$ quelconque inférieur à $10^{15}$.

Soit $\Pplus(M)$ le plus grand facteur premier du nombre
entier positif $M$.  Il est montré dans 
(\cite{NICAA,NICTH})
que 
\begin{equation}\label{lab5}
\Pplus(g(n)) \sim \log g(n) \sim \sqrt{n \log n},\qquad n \to +\infty
\end{equation}
et dans \cite{MNRMC}
\begin{equation}\label{}
\Pplus(g(n)) = \sqrt{n\log n}\left(1+\frac{\log_2 n +
    \bigo{1}}{2\log n} \right),\qquad n \to +\infty
\end{equation}
ainsi que la borne effective
\begin{equation}\label{}
\Pplus(g(n)) \le 2.86 \sqrt{n \log n},\quad n \ge 2.
\end{equation}
Cette dernière majoration a été améliorée par Grantham
(cf. \cite{GRA})~:
\begin{equation}\label{}
\Pplus(g(n)) \le 1.328 \sqrt{n \log n},\quad n \ge 5.
\end{equation}

L'objet de cet article est de donner des estimations plus précises
de $\Pplus(g(n))$. 

Nous minorerons $\Pplus(g(n))$ au paragraphe \ref{parmino};
cette minoration est basée sur la remarque suivante
(cf. lemme \ref{parminolemm1}): soit $q$ le nombre premier
suivant $\Pplus(g(n))$ (donc $q$ ne divise pas $g(n)$)
et soit $\lb$ un nombre premier tel que $\lb^{\al}$
divise $g(n)$ mais pas $\lb^{\al+1}$, alors $\lb^{\al} \le 2q$.
Ainsi la contribution des petits facteurs premiers
dans $g(n)$ est faible, et l'équivalence \eqref{equilandau}
impose que $\Pplus(g(n))$ ne soit pas trop petit.

La majoration, qui sera traitée au paragraphe \ref{parmajo}, est plus
délicate. Il faut en effet montrer que le nombre de facteurs premiers
inférieurs à $P = \Pplus(g(n))$ ne divisant pas $g(n)$ n'est pas trop
grand.  Pour cela nous utiliserons la méthode de Grantham \cite{GRA}.
On remarque (cf. lemme \ref{lemnicolas2}) qu'au plus un nombre premier
inférieur à $P/2$ ne divise pas $g(n)$. Ensuite on construit une suite
croissante $(\gm_k)$ avec $\gm_0 = 0$ telle qu'au plus un nombre
premier de l'intervalle $\intfd{\gm_k P}{\gm_{k+1} P}$ ne divise pas
$g(n)$.  La construction d'une telle suite est expliquée aux
paragraphes \ref{gsuite1} et \ref{gsuite2}.

Dans le paragraphe \ref{pplusvslog}, nous montrerons, à l'aide
de théorèmes d'oscillation en théorie des nombres premiers,
qu'il existe une infinité d'entiers $n$ tels
que $\Pplus(g(n)) > \log g(n)$ et une infinité d'entiers $n$
tels que $\Pplus g(n) < \log g(n)$.

Pour cela nous rappellerons la définition des nombres
$\ell$--superchampions qui jouent pour la fonction $\ell$ le même rôle
que les \og superior highly composite numbers \fg\ de Ramanujan
(cf. \cite{RAMCOL}, \S\, 32) pour la fonction
nombre de diviseurs.

Dans le paragraphe \ref{tabulations} nous introduirons
les tables numériques sur la distribution des nombres
premiers qui figurent en annexe. Des tables plus longues
sont déposées sur le site \url{http://math.univ-lyon1.fr/~deleglis/calculs.html}.

Dans le théorème \ref{t3} les nombres $5.54$ et $10.8$
pourraient être améliorés au prix d'un alourdissement des calculs numériques.
Ces constantes ont été déterminées en calculant $\Pplus(g(n))$ pour
tous les $n$ jusqu'à $10^6$. Cette borne pourrait être augmentée. Il
est en effet possible d'adapter l'algorithme de \cite{DNZ}
de fa\c con à donner rapidement, pour deux nombres
$\ell$--superchampions consécutifs $N$ et $N'$ le maximum et le minimum
de $\Pplus(g(n))$ dans l'intervalle $[\ell(N),\ell(N')]$.

\subsection{Notations}\label{secnotations}
Nous utiliserons les notations classiques suivantes.
\begin{description}
\item{1)}
Pour $i \ge 1$, $p_i$ est le $i\iem$ nombre premier.
\item{2)}
$\pi(x) = \sum_{p \le x} 1$ est le nombre des nombres premiers $\le x$.
\item{3)}
$\theta(x)$ et $\psi(x)$ sont les fonctions de Chebyshev
\begin{equation}\label{defthetapsi}
\theta(x) = \sum_{1 \le i \le \pi(x)} \log p_i
\qtx{et}
\psi(x)  =  \sum_{ p^{m} \le x} \log p = \sum_{k \le K} \theta(x^{1/k})
\end{equation}
o\`u $K$ est le plus petit entier tel que $x^{1/K} < 2$.
\item{4)}
$\Theta$ est la borne supérieure des parties réelles des zéros
de la fonction $\zeta$ de Riemann.
Sous l'hypothèse de Riemann, on a $\Theta = 1/2$.
\item{5)}
$\log_2 x = \log\log x$ et pour $k \ge 3$, $\log_k(x) = \log(\log_{k-1}(x)).$
\item{6)}
$\Li(x)$, le logarithme intégral de $x$, est défini pour $x > 1$ par
\[
\Li(x) = \lim_{\veps \to 0^+}%
\int_{0}^{1-\veps} + \int_{1+\veps}^{x} \frac{\d t}{\log t}%
= \gm + \log_2 x + \sum_{n=1}^{+\infty} \frac{(\log x)^n}{n n!},
\]
o\`u $\gm = 0.577\dots $ est la constante d'Euler. 
\item{7)}
Soit $f$ une fonction de la variable réelle $x$, continue par morceaux.
On note $f(x^-)$ et $f(x^+)$ les nombres définis par
\[
f(x^-) = \lim_{\stackrel{t \to x}{t < x}} f(t)
\qtx{ et }
f(x^+) = \lim_{\stackrel{t \to x}{t > x}} f(t).
\]
\end{description}

\subsection{Rappel des estimations de $\log g(n)$}\label{rappelsgn}

Nous utiliserons les résultats effectifs suivants, démontrés dans
\cite{MAS} et \cite{MNRMC}.

\begin{equation}\label{L1}
\log g(n) \le 1.053139976709\dots \sqrt{n \log n},\quad n \ge 1
\end{equation}
avec égalité pour $n = 1\,319\,766$.

\begin{eqnarray}
\label{L2}
\log g(n) &\ge& \sqrt{n \log n},%
\hphantom{\left(1+\frac{\log_2 n - 0.975}{2\log n} \right)}\quad n \ge 906. \\
\label{L3}
\log g(n) &\le& 
\sqrt{n \log n}\left(1+\frac{\log_2 n - 0.975}{2\log n} \right),\quad n \ge 3
\\
\label{L4}
\log g(n) &\ge&
\sqrt{n \log n}\left(1+\frac{\log_2 n - 1.18}{2\log n}\right),
\quad\; n \ge 899\,059
\end{eqnarray}
Notons que la majoration \eqref{L4} est meilleure que \eqref{L3}
pour $n \ge 68 \,745\,487$.

Nous aurons aussi besoin des résultats asymptotiques de \cite{MNRAA}~:

\noindent
Il existe une constante $a > 0$ telle que
\begin{equation}\label{asy}
\log g(n) = \sqrt{\Li^{-1}(n)} + \bigo{\sqrt{n}e^{-a\sqrt{\log n}}}.
\end{equation}
Si $\Theta < 1$, on a
\begin{equation}\label{asyTH}
\log g(n) = \sqrt{\Li^{-1}(n)} + \bigo{(n \log n)^{\Theta/2}}.
\end{equation}

\section{Fonctions portant sur les nombres premiers}\label{tabulations}

\subsection{Les fonctions $\theta$ et $\psi$ de Chebyshev}

\subsubsection{Encadrements effectifs}

Nous utiliserons les résultats suivants.

\begin{equation}\label{T1}
\theta(x) < x \text{ pour } x \le 8\cdot10^{11}.
\end{equation}
Schoenfeld (cf. \cite{SCH76} p. 360) mentionne que R. P. Brent a
vérifié \eqref{T1} pour $x < 10^{11}$. P. Dusart (cf. \cite{DUSMC}) a
calculé $\theta(x)$ jusqu'à $8.10^{11}$ et a établi la majoration
suivante (qui améliore le résultat de \cite{SCH76} p. 360
o\`u la majoration $\theta(x) < 1.000\,081\,x$ est prouvée pour $x > 0$)
\begin{equation}\label{T2}
\theta(x) < x + \frac{1}{36\,260} x \le 1.000\,028\,x\quad (x > 0).
\end{equation}
L'encadrement ci-dessous figure dans (\cite{RS62}, Th. 18)
\begin{equation}\label{T3}
x - 2.06\sqrt x < \theta(x) < x,\quad(0 < x \le 10^8).
\end{equation}
L'inégalité suivante (cf. \cite{DUSMC}) améliore le résultat de
\cite{SCH76}
où le coefficient $0.2$ était $8.072$
\begin{equation}\label{T4}
\abs{\theta(x)-x} \le \frac{0.2\, x}{\log^2 x},\quad (x \ge 3\,594\,641).
\end{equation}
On trouvera les inégalités suivantes dans (\cite{RS62}, Th. 13)
\begin{equation}\label{T5}
\theta(x) \le \psi(x) \le \theta(x) + 1.42620\,\sqrt x,
\qquad x > 0.
\end{equation}

\subsubsection{La fonction $\thmin$}\label{sectionthmin}
Nous aurons aussi besoin de minorer $\dfrac{\theta(x)}{x}$ sur des
intervalles de la forme $\intfg{y}{+\infty}$.  Notons
\begin{equation}\label{thmin}
\thmin(y) =
\inf_{x \ge y} \dfrac{\theta(x)}{x}\cdot
\end{equation}
La fonction $\thmin$ est une fonction en escalier croissante et
continue à droite.  Puisque $\theta$ est constante sur tout intervalle
$\intfg{p_{i-1}}{p_{i}}$, le rapport $\theta(x)/x$ décroit sur cet
intervalle, avec pour borne inférieure $\theta(p_{i-1})/p_{i}$.  Il en
résulte que
\begin{equation}\label{thminpi}
\thmin(y) = \inf_{p_{i} > y}\ \frac{\theta(p_{i-1})}{p_{i}}
\end{equation}
et, si l'on définit l'indice $i_y$ par
$p_{i_y-1} \le y < p_{i_y}$, 
on a 
\begin{equation}\label{piy}
\thmin(y) = 
\inf\set{\frac{\theta(p_{i_y-1})}{p_{i_y}},\frac{\theta(p_{i_y})}{p_{i_y+1}},\cdots}
= \min\set{\frac{\theta(p_{i_y-1})}{p_{i_y}},\thmin(p_{i_y})}.
\end{equation}
Par la formule
\eqref{piy}, $\thmin$ est constante 
sur \dsm{\intfg{p_{i_y-1}}{p_{i_y}}} ; ainsi les points de
discontinuité de $\thmin$ sont des nombres premiers que l'on
appelle nombres $\thmin$--champions. 

Pour tout nombre premier $p$, désignons par $\starp$ le nombre premier
précédant $p$.  
Soit $p$ un nombre $\thmin$--champion.  Puisque $\thmin$ est
croissante, au point de discontinuité $p$, on a
\begin{equation}\label{pp-}
\thmin(p) > \thmin(p^-).
\end{equation}
Lorsque  $y$ tend vers $p$ par valeurs inférieures, on a $p=p_{i_y}$,
et les formules \eqref{piy} et \eqref{pp-} donnent
\begin{equation}\label{thp-}
\thmin(p^-) =  \min\set{\frac{\theta(\starp)}{p},\,\thmin(p)}
= \frac{\theta(\starp)}{p} < \thmin(p).
\end{equation}
Si $p < q$ sont deux nombres $\thmin$--champions consécutifs,
la fonction $\thmin$ est constante sur $\intfg{p}{q}$ et,
pour $p \le y < q$, on a
\begin{equation}\label{yqpp-}
\thmin(y) = \thmin(p) = \thmin(q^-) = \frac{\theta(\starq)}{q}
< \thmin(q).
\end{equation}
Ainsi, $p$ est le plus grand nombre premier inférieur à $q$ et
vérifiant
\begin{equation}\label{pq-}
\thmin(p^-) = \frac{\theta(\starp)}{p} < \frac{\theta(\starq)}{q} \cdot
\end{equation} 

La table \ref{tablethmin} contient
les premiers $\thmin$--champions $p$ et leurs records,
$\thmin(p)$, arrondis par défaut.
Pour $y \ge 2$ la valeur $\thmin(y)$ est donnée par  
\begin{equation}\label{deltavalue}
\thmin(y) = \thmin(p) \text{ avec  } p 
\text{ le plus grand champion $\le y$ de }\thmin.
\end{equation}
Pour calculer la table \ref{tablethmin}, on observe d'abord
que pour $x \ge 10\,000$ on a
\dsm{\frac{\theta(x)}{x} \ge 0.9794}.
Pour $10^4 \le x \le 10^8$ cela résulte de \eqref{T3}~:
\begin{equation}\label{9794}
\frac{\theta(x)}{x} > 1 - \frac{2.06}{\sqrt x} 
\ge 1-\frac{2.06}{\sqrt{10^4}} =  0.9794,
\end{equation}
tandis que, pour $x > 10^8$, \eqref{T4} implique
\begin{equation}\label{9994}
\frac{\theta(x)}{x} \ge 1-\frac{0.2}{\log^2 x} 
\ge 1-\frac{0.2}{\log^2 10^8} \ge 0.9994.
\end{equation}
Le plus grand nombre premier $P < 10\,000$ vérifiant
\dsm{\theta(\starP)/P < 0.9794} est $P = 7477$.
Par \eqref{thmin}, \eqref{9794} et \eqref{9994}, on a
$\thmin(10000) \ge 0.9794$, ce qui, par \eqref{thminpi},
entraîne $\theta(\starp)/p \ge 0.9794$ pour $p \ge 10000$.
Le choix $P = 7477$ implique $\theta(\starp)/p \ge 0.9794$
pour $P < p < 10000$. On a donc, par \eqref{thminpi},
$\thmin(P) \ge 0.9794$ et, par \eqref{piy},
\[
\thmin(P^-) = \thmin(\starP) = 
\min\left(\frac{\theta(\starP)}{P},\thmin(P) \right)
= \frac{\theta(\starP)}{P} < \thmin(P),
\] 
ce qui montre que $P = 7477$ est un $\thmin$--champion.

Puis, par récurrence descendante, à partir d'un champion $q$,
on détermine le champion $p$ précédant $q$ par la règle
\eqref{pq-}.

\subsubsection{La fonction $\theta_d$}
\begin{definition}
On définit la fonction $\theta_d$ pour $y \ge  1$ par
\[
\theta_d(y) = \sup_{x \ge y} \abs{\frac{\theta(x)}{x}-1}\log^2 x,
\]
de sorte que, pour tout $x \ge y$ on a
\begin{equation}\label{propthetad}
\abs{\frac{\theta(x)}{x}-1} \le \frac{\theta_d(y)}{\log^2 x}\cdot
\end{equation}
\end{definition}

Le résultat \eqref{T4} de Dusart entraine~:

\begin{lem}\label{lemm1}
Pour $y \ge 3\,594\,641$ on a $\theta_d(y) \le 0.2.$
\end{lem}

\begin{lem}\label{lemttd}
Posons $p_0=1$ et soit $i \ge 0$ tel que $p_i < p_{i+1} < 8\cdot
10^{11}$. La fonction $x \mapsto \abs{\frac{\theta(x)}{x}-1} \log^2 x$
est croissante sur l'intervalle \dsm{[p_i,\,p_{i+1})} et  l'on a
\[
\sup_{p_i \le x < p_{i+1}} \abs{\frac{\theta(x)}{x}-1} \log^2 x 
= \frac{p_{i+1}-\theta(p_i)}{p_{i+1}}\log^2 p_{i+1}\cdot
\]
\end{lem}

\begin{proof}
Soit $x$ vérifiant $p_i \le x < p_{i+1}$ et posons 
$f(x) = \abs{\dfrac{\theta(x)}{x}-1} \log^2 x$. 
Par \eqref{T1} on a  
\[
f(x) = \left(1-\frac{\theta(p_i)}{x}\right)\log^2 x
\]
et
\dsm{f'(x) = \frac{\log x}{x^2} f_1(x)} avec 
$f_1(x) = \theta(p_i) (\log x-2) + 2x$.
La fonction $f_1$ est croissante, donc $f_1(x) \ge f_1(p_i) > 0$
pour $p_i \ge 11 > e^2$. Le calcul de 
$f_1(p_i)$ pour $0 \le i \le 4$ montre que $f_1(p_i) > 0$
pour tout $i \ge 0$. Ainsi $f$ est croissante sur l'intervalle
$[p_i,\,p_{i+1})$ ce qui prouve le lemme.
\end{proof}

Si $p$ est un nombre premier, posons 
\dsm{F(p) = \left(1-\frac{\theta(^\star p)}{p}\right)\log^2 p}
où $^\star p$ désigne le nombre premier précédant $p$, avec la
convention $^\star 2=1$. Le calcul montre que $P=3\,594\,641$ est premier
et que l'on a $F(P)=0.200386\dots > 0.2$. 

Il résulte des lemmes \ref{lemm1} et  \ref{lemttd} que l'on a pour $y < P$
\begin{equation}\label{thdF}
\theta_d(y) = \max_{y < p \le P} F(p).
\end{equation}
La formule \eqref{thdF} montre que, sur l'intervalle $[1,P)$,
  $\theta_d$ est une fonction en escalier décroissante et continue à
  droite.  Appelons $\theta_d$--champions les nombres premiers $p\le
  P$ qui sont points de discontinuité de $\theta_d$. Convenons aussi
  que $1$ est un $\theta_d$--champion. Le nombre $P$ est aussi un
$\theta_d$--champion~: en effet, par le lemme \ref{lemm1}, on a 
$\theta_d(P) \le 0.2$ tandis que $\theta_d(P^{-}) = F(P) > 0.2$.

Si $p$ et $q$ sont deux nombres $\theta_d$--champions consécutifs
vérifiant $1\le p < q \le P$, on a donc pour $p \le y <q$~:
\[
\theta_d(y) =\theta_d(p) = \theta_d(q^-)= F(q) >  \theta_d(q)
\]
et, si $p \ne 1$, $p$ est le plus grand nombre premier inférieur à $q$
et vérifiant
\begin{equation}\label{Fp>Fq}
F(p) > F(q).
\end{equation}
Le nombre $P=3\,594\,641$ est un $\theta_d$--champion. Par la formule 
\eqref{Fp>Fq}, si $q\le P$ est un nombre $\theta_d$--champion, le
champion précédant $q$ est le plus grand nombre premier $p$ vérifiant
$p < q$ et $F(p) > F(q)$.

On calcule ainsi par récurrence descendante les nombres
$\theta_d$--champions jusqu'à $59$. Mais pour tout nombre premier $p <
59$, on a $F(p) < F(59)$. La fonction $\theta_d$ est donc constante
sur l'intervalle $[1,59)$ et vaut $F(59)=3.9648\ldots$. On a donc,
pour tout $x \ge 1$, 
\[
\left| \frac{\theta(x)}{x}-1\right| \log ^2(x)\leq \theta_d(x) \le
\theta_d(59^-)=F(59) < 3.965.
\]
La table \ref{tableThetad} contient la liste des premiers nombres 
$\theta_d$--champions.

\subsection{Terme d'erreur dans le théorème des nombres premiers}

Soit $\theta$ la fonction de Chebyshev définie en \eqref{defthetapsi}.

\begin{enumerate}
\item
Il existe $a > 0$ tel que
\begin{equation}\label{TNP1}
\theta(x) = x  + \bigo{x\exp\big(-a\sqrt{\log x}\,\big)}
\end{equation}
\item
Si $\Theta < 1$, on a
\begin{equation}\label{TNP2}
\theta(x) = x +  \bigo{x^{\Theta} \log^2 x}.
\end{equation}
\item
Si l'hypothèse de Riemann est vraie, on a
\begin{equation}\label{TNP3}
\abs{\theta(x)-x} \le \frac{1}{8\pi} \sqrt x \log^2 x, \qquad x \ge 599.
\end{equation}
\end{enumerate}
Les points 1. et 2. se trouvent dans les traités de théorie analytique
des nombres, par exemple \cite{INGHAM32} ou \cite{EMF}. Le point 3.
est prouvé dans \cite{SCH76}, p. 337.

\subsection{Écarts entre nombres premiers}\label{ecartsp}

\begin{enumerate}
\item
Nous utiliserons le résultat suivant de Baker, Harman et Pintz
(\cite{BHP} p. 562).
Soit $\delta = 0.525$, alors, pour $x$ assez grand, on a
\begin{equation}\label{BHP}
\pi(x+x^{\delta})-\pi(x) \ge \frac{9}{100} \frac{x^{\delta}}{\log x},
\end{equation}
ce qui implique que, pour $p_i$ assez grand, $p_{i+1} \le p_i +
p_i^{\dt}$, et donc l'existence de $a$ tel que
\begin{equation}\label{bhppi}
p_{i+1} \le p_i + a p_i^{\dt},\qquad i \ge 1.
\end{equation}

\item
De façon effective, Dusart a démontré dans \cite{DUS10}, proposition 6.8,
que, pour $x \ge 396\, 738$, l'intervalle 
\begin{equation}\label{label21}
\intf{x}{x+\frac{x}{25 \log^2 x}}
\end{equation}
 contient un nombre premier. Cela entraine pour 
$p_i \ge 396\,833 = p_{33\,609}$, 
\begin{equation}\label{duspi}
p_{i+1} \le p_i + \frac{p_i}{25\,\log^2 p_i}\cdot
\end{equation}
\item
Si l'hypothèse de Riemann est vraie la formule \eqref{TNP3}
permet de montrer que, pour $x \ge 599$, l'intervalle
\dsm{\intf{x-\sqrt x \log^2 x/(4\pi)}{x}} contient un nombre
premier (cf. \cite{RAMSAO}, (1)).

Dans l'article \cite{CRA}, toujours sous l'hypothèse de Riemann,
Cramér démontre qu'il existe $b$ tel que l'intervalle
$\intf{x}{x+b\sqrt x \log x}$ contienne un nombre premier.
Ramaré et Saouter ont rendu ce résultat effectif en montrant
dans (\cite{RAMSAO},  th. 1) que, pour $x \ge 2$,
l'intervalle
\begin{equation}\label{craint}
\intf{x-\frac{8}{5}\sqrt x \log x}{x} 
\end{equation}
contient un nombre premier, ce qui entraine que, pour $p_i \ge 3$,
\begin{equation}\label{lab18}
p_{i-1} \ge p_i - \frac85 \sqrt p_i \log p_i.
\end{equation}

\item
Dans l'article \cite{CRA}, la \og Conjecture de Cramér \fg\ est
énoncée comme suit
\begin{equation}\label{conjcra}
p_{i+1} - p_i = \bigo{\log^2 p_i}.
\end{equation}
Cette conjecture est étayée par les calculs numériques (cf. par
exemple \cite{Nicy}).
\end{enumerate}

\subsection{Les fonctions $\eta_k$}\label{etasection}
\subsubsection{Définition}

Soit $k \ge 1$ un nombre entier.
Par le théorème des nombres premiers, le rapport $p_{i-k}/p_{i}$ tend
vers $1$ quand $i \to +\infty$. Pour $i_0 \ge k+1$ il n'y a donc
qu'un nombre fini d'entiers $i$ tels que 
\dsm{\frac{p_{i-k}}{p_{i}} \le \frac{p_{i_0-k}}{p_{i_0}} < 1}, 
et la définition suivante a un sens.

\begin{definition}\label{defetak}
On définit la fonction $\eta_k$ sur l'intervalle
$\intfg{p_k}{+\infty}$ par
\begin{equation}\label{eta1}
\eta_k(x) =  \min\set{\frac{p_{i-k}}{p_{i}} \Big| p_i > x}.
\end{equation}
\end{definition}

Il résulte de \eqref{eta1} que $\eta_k$ est une fonction
en escalier croissante dont les points de discontinuté sont
des nombres premiers appelés \emph{nombres $\eta_k$--champions}.
Par convention, $p_k$ est un nombre $\eta_k$--champion. Plus
précisément, si $p_{i'}$ et $p_{i\sgg}$ sont deux nombres
$\eta_k$--champions consécutifs, on a pour $p_{i'} \le x < p_{i\sgg}$
\[
\eta_k(x) = \eta_k(p_{i'}) = \frac{p_{i\sgg-k}}{p_{i\sgg}}\cdot
\]
Le nombre premier $p_i$ est un nombre $\eta_k$--champion si l'on
a ou bien $i=k$ ou bien $i > k$ et
\begin{equation}\label{etakch}
\eta_k(p_{i-1}) < \eta_k(p_i).
\end{equation}
\begin{lem}\label{lemetak}
Soit $x \ge p_k$.
Pour tout $y \ge x$, l'intervalle $\intfd{ \eta_k(x)y}{y}$
contient au moins $k$ nombres premiers, et $\eta_k(x)$ est le
plus grand réel $\lb$ tel que \dsm{\intfd{\lb y}{y}}
contienne au moins $k$ nombres premiers pour tout $y \ge x$.
\end{lem}

\begin{proof}
Notons $\eta = \eta_k(x)$.  Si l'intervalle $\intfd{\eta y}{y}$ ne
contient pas $k$ nombres premiers, soit $p_{i}$ le plus petit nombre
premier qui est strictement plus grand que $y$. Alors
\[
p_{i-k} \le \eta y < y < p_i,
\]
et donc
\dsm{\eta_k(x) = \frac{\eta y}{y} > \frac{p_{i-k}}{p_{i}}} ce qui
est absurde car $p_i > x$. 

Réciproquement, posons $\eta_k(x) = \frac{p_{i_0-k}}{p_{i_0}}$
avec $p_{i_0} > x$ et supposons que l'intervalle $\intfd{\lb y}{y}$
contienne $k$ nombres premiers pour tout $y \ge x$.
Choisissons $y$ tel que $p_{i_0} > y > \max(p_{i_0-1},x)$.
On doit avoir $\lb y < p_{i_0-k}$, ce qui, en faisant tendre $y$
vers $p_{i_0}$, donne $\lb \le \eta_k(x)$.
\end{proof}

\subsubsection{Minoration de $\eta_k$}

\begin{lem}\label{lemink}
 Soit $i_0 \ge 1$ un entier et 
$f~:\ \intfg{p_{i_0}}{+\infty} \rightarrow \R$ une
fonction croissante. On suppose que l'on a la majoration
\begin{equation}\label{pif}
p_{i+1} \le p_i + f(p_{i+1}),\qquad i \ge i_0.
\end{equation}
Soit $k \ge 1$ et $x \ge p_{i_0+k-1}$. Si la fonction 
$t \mapsto f(t)/t$ est décroissante pour $t \ge x$,
on a
\begin{equation}\label{etak>}
\eta_k(x) \ge 1 - k\frac{f(x)}{x}\cdot
\end{equation}
\end{lem}

\begin{proof}
Soit $i \ge 1$ tel que $p_i > x$. Puisque
$x \ge p_{i_0+k-1}$, on a $i \ge i_0+k$. Ecrivons
l'inégalité \eqref{pif} pour $i-1$, $i-2,\, \dots,i-k$; on
obtient
\begin{eqnarray*}
  p_i &\le& p_{i-1} + f(p_i) \\
  p_{i-1} &\le& p_{i-2} + f(p_{i-1}) \le p_{i-2} + f(p_i) \\
&\dots&\\
  p_{i-k+1} &\le& p_{i-k} + f(p_{i-k+1}) \le p_{i-k} + f(p_i).
\end{eqnarray*}
En ajoutant ces inégalités il vient
\[
p_i \le p_{i-k} + k f(p_i)
\]
et, par la décroissance de $f(t)/t$,
\[
\frac{p_{i-k}}{p_i} \ge 1 - k \frac{f(p_i)}{p_i}
\ge 1 - k \frac{f(x)}{x}
\]
ce qui, d'après la définition \eqref{eta1} de $\eta_k$, prouve
\eqref{etak>}.
\end{proof}

\begin{prop}\label{propmin}
\mbox{}
\begin{enumerate}
\item
Il existe $a > 0$ tel que,  pour $x \ge p_k$, ,on ait
\[
\eta_k(x) \ge 1-k \frac{a}{x^{0.475}}\cdot
\]
\item
Soit $i_0 = 33\,609$. Pour $i$ voisin de $i_0$ les valeurs
de $p_i$ sont

\begin{center}
\begin{tabular}{|r|r|r|r|r|r|}
$i=$& $33\, 608$ & $33\,609$& $33\,610$& $33\,611$& $33\,612$\\ 
\hline
$p_i=$& $396\,733$ & $396\,833$ & $396\,871$ & $396\,881$& $396\,883$\\
\hline
\end{tabular}
\end{center}
Pour $x \ge  p_{i_0+k-1}$ on a
\begin{equation}\label{dus25}
\eta_k(x) \ge 1-\frac{k}{25\log^2 x}\cdot
\end{equation}
\item
Si l'hypothèse de Riemann est vraie, on a pour $x \ge \max(p_k,e^2)$,
\begin{equation}\label{lem3.3}
\eta_k(x) \ge 1-\frac{8k}{5}\frac{\log x}{\sqrt x}\cdot
\end{equation}
\item
Sous la conjecture de Cramér \eqref{conjcra}, il existe $a > 0$ tel
que, pour  $x \ge \max(p_k,e^2)$ on ait
\begin{equation}\label{etacra}
\eta_k(x) \ge 1 - ka\frac{\log^2 x}{x}\cdot
\end{equation}

\end{enumerate}
\end{prop}

\begin{proof}
\mbox{}
\begin{enumerate}
\item
On applique le lemme \ref{lemink} avec $i_0 = 1$ et $f(t) = at^{\dt}$
où $\dt = 0.525$ et $a$ est la constante donnée en \eqref{bhppi}.

\item
On choisit $i_0 = 33\,609$ et $f(t) = t/(25\log^2(t))$.
Par \eqref{duspi}, l'hypothèse \eqref{pif} du lemme \ref{lemink}
est vérifiée, et l'application de ce lemme donne le résultat.

\item
Cette fois, on applique le lemme \ref{lemink} avec $i_0 = 1$
et $f(t) = \frac{8}{5}\sqrt{t} \log t$. L'hypothèse \eqref{pif}
résulte de  \eqref{lab18} et la fonction $f(t)/t$ est décroissante
pour $t \ge e^2$.

\item
Par la conjecture de Cramér \eqref{conjcra}, il existe $a$ tel que
$p_{i+1} \le p_i + a\log^2 p_i$ pour tout $i \ge 1$. On choisit
donc $i_0 = 1$ et $f(t) = a \log^2 t$ dans le lemme \ref{lemink}
qui donne le résultat en remarquant que $f(t)/t$ est
décroissante pour $t \ge e^2$.
\end{enumerate}
\end{proof}

\subsubsection{Tabulation des valeurs de $\eta_k$.}\label{tabetak}
On trouvera en annexe les tables \ref{tableeta1}, \ref{tableeta2}
et \ref{tableeta3} qui permettent le calcul de $\eta_1(x)$, $\eta_2(x)$
et $\eta_3(x)$.
Expliquons le calcul des valeurs de $\eta_3$; on obtient la table
de $\eta_k$ pour $k \ne 3$ de façon similaire.

Soit $x_0 = 10^6$. Par la proposition \ref{propmin} on a
\[
\eta_3(x_0) > 1 - \frac{3}{25\,\log^2(x_0)} > 0.99937.
\]
Par la croissance de la fonction $\eta_3$, on a donc pour tout
$x \ge 10^6$, $\eta_3(x) > 0.99937.$ De la définiton
\eqref{eta1} de $\eta_3$ on déduit
\begin{equation}\label{eta1pi}
\eta_3(p_i) = \min\set{\frac{p_{i-2}}{p_{i+1}}, \frac{p_{i-1}}{p_{i+2}},\dots}
= \min\set{\frac{p_{i-2}}{p_{i+1}},\, \eta_3(p_{i+1})}.
\end{equation}
 Soit $i_0$ le plus grand indice tel que $p_{i_0+1 }< 10^6$ et
$p_{i_0-2}/p_{i_0+1} \le 0.99937$ (le calcul donne 
$i_0 = 15\,929$, $p_{i_0-2} = 175\,141$, $p_{i_0-1}= 175\,211$,
$p_{i_0}= 175\,229$, $p_{i_0+1} = 175\,261$).
On a
\begin{equation}\label{imp937}
p_i > p_{i_0} \implies \eta_3(p_i) > 0.99937
\end{equation}
et, par \eqref{eta1pi},
\begin{equation}\label{imp931}
\eta_3(p_{i_0}) = \frac{p_{i_0-2}}{p_{i_0+1}} =
\frac{175\,141}{175\,261} =  0.99931\dots
\end{equation}
ce qui, par \eqref{etakch}, prouve que $p_{i_0+1} = 175\,261$
est un nombre $\eta_3$--champion.

Par la formule \eqref{eta1pi}, on calcule ensuite $\eta_3(p_i)$ pour
$i=i_0-1$, $i_0-2$, $\dots, 1$ et, par \eqref{etakch}, on en déduit
les nombres $\eta_3$--champions.

\subsubsection{La fonction $\dt_3$}

\begin{definition}
On d\'efinit la fonction $\dt_3$, pour $y \ge p_3 = 5$, par
\[
\dt_3(y) = \sup_{x \ge y} (1-\eta_3(x))\log^2 x.
\]
Pour tout $x \ge y$ on a donc
\begin{equation}\label{pd3}
1-\eta_3(x) \le \frac{\dt_3(y)}{\log^2(x)}\cdot
\end{equation}
\end{definition}
La fonction $\dt_3$ est d\'ecroissante.
La minoration \eqref{dus25} de la proposition \ref{propmin}  entraine
\begin{equation}\label{etiq3/25}
\dt_3(396\, 881) \le \frac{3}{25} =  0.12.
\end{equation}

 On a remarqu\'e, lors de la d\'emonstration du lemme \ref{lemetak}, que 
$1-\eta_3$ est constante sur chaque intervalle
$\intfg{p_{i-1}}{p_{i}}$ ($i \ge k+1$).
La fonction $x \mapsto (1-\eta_3(x))\log^2 x$ 
est donc strictement croissante sur l'intervalle  $[p_{i-1}, p_{i})$,
et la borne sup\'erieure de ses valeurs sur cet intervalle
est $(1-\eta_3(p_{i-1}))\log^2(p_{i})$.

Il en r\'esulte que, pour $y \ge 5$,
\begin{equation}\label{d3m}
\dt_3(y) = \max_{p_{i} > y} G(p_i)
\qtx{avec}
G(p_i) = \big(1-\eta_3(p_{i-1})\big)\log^2(p_{i}).
\end{equation}
La formule \eqref{d3m} montre que $\dt_3$ est une fonction en escalier
d\'ecroissante continue \`a droite.
Appelons $\dt_3$--champions les nombres premiers $p$ qui sont points
de discontinuit\'e de $\dt_3$. Nous dirons aussi que $5$ est un 
$\dt_3$--champion.
Si $p < q$ sont deux nombres $\dt_3$--champions cons\'ecutifs, on a pour
$p \le y < q$
\[
\dt_3(y) = \dt_3(p) = \dt_3(q^-) = G(q) > \dt_3(q)
\]
et, si $p > 5$, $p$ est le plus grand nombre premier inf\'erieur \`a
$q$ tel que $G(p) > G(q)$.

Il r\'esulte de \eqref{d3m} et \eqref{etiq3/25} que,
pour $p_i > 396\,881$, on a
$G(p_i) \le \dt_3(396\,881) < 0.12$.
Lorsque $175\,261 < p_i \le 396\,881$,
l'implication \eqref{imp937}  donne 
$\eta_3(p_{i-1}) \ge 0.99937$ et 
$G(p_i) \le 0.00063 \log^2(396\,881) < 0.12$.

\`A l'aide de la table construite au paragraphe \ref{tabetak}, pour $7
\le p_i \le 175\,261$, on sait calculer $\eta_3(p_{i-1})$ et
$G(p_i)$. On recherche alors le plus grand nombre premier $p_i <
396\,881$ pour lequel on a $G(p_i) > 0.12$. C'est $p_i =
88\,211$ qui est un nombre $\dt_3$--champion. Si $q$ est un
$\dt_3$--champion, le nombre $\dt_3$--champion pr\'ec\'edant $q$ est le
plus grand nombre premier $p$ tel que $G(p) >
G(q)$.  En \'enum\'erant les nombres premiers $p$ inf\'erieurs \`a
$88\,211$ on peut ainsi dresser la table \ref{tabledelta3} des nombres
$\dt_3$--champions.

\section{Facteurs premiers de $g(n)$ et $g$--couples}\label{gsuite1}

\begin{notation}
Pour tout intervalle réel $I$ et tout réel $\lb$ on note
$\lb I$ l'intervalle défini par $\lb I = \setof{\lb x}{x \in I}$.
\end{notation}

Les deux lemmes suivants, dont la démonstration est facile,
se trouvent respectivement dans 
(\cite{NICAA}, p. 142) et (\cite{GRA}, p. 408).

\begin{lem}\label{lemnicolas}
Si $q$ divise $g(n)$ et si $p', \sg{p}, q$ sont premiers distincts
et $q \ge p' + \sg{p}$ alors $p'$ ou  $\sg{p}$ divise $g(n)$.
\end{lem}

\begin{lem}\label{lemnicolas2}
Si $q$ est un diviseur premier de $g(n)$, il existe au  plus un nombre
premier $\le q/2$ qui ne divise pas $g(n)$.
\end{lem}

\begin{definition}
Soit $\gm, \gm'$ avec $0 < \gm < \gm'< 1$ et $\gm' <
\dfrac{1+\gm^2}{2}$.  On définit
\begin{equation}\label{defalphabeta}
\al = 2\gm'-1 \qtx{et} \bt = \gm^2.
\end{equation} 
On a alors $\al < \bt< \gm < \gm'$ et le couple d'intervalles $(I,J)$
défini par
\[
I = I(\gm,\gm') = \intfd{\al}{\beta}
\text{ et }
J = J(\gm,\gm') = \intfd{\gm}{\gm'}.
\]
est appelé le g-couple associé à $(\gm,\gm')$.
\end{definition}

En remarquant que \dsm{\gm = \sqrt{\bt}} et que \dsm{\gm'=\frac{1+\al}{2}}
le lemme suivant est le lemme 2 dans \cite{GRA}, dont nous rappelons
la preuve.

\begin{lem}\label{lemgrantham}
Soit $(I,J)$ un g-couple,
$n \ge 1$ et $q$ un facteur premier de $g(n)$.
Si \dsm{q I} contient au moins un diviseur
premier de $g(n)$ tous les nombres premiers
appartenant à $qJ$, à l'exception d'au plus un,
sont des diviseurs de $g(n)$.
\end{lem}

\begin{proof}
Par hypothèse, $qI$ contient un diviseur premier $q'$ de $g(n)$ et donc
$\al q < q' \le \bt q < q$. S'il existait dans $qJ$ deux nombres
premiers
$p$ et $p'$ ne divisant pas $g(n)$, on aurait
$\gm q < p < p' \le \gm' q$. En posant $M = \dfrac{pp'}{qq'}g(n)$,
il viendrait $\ell(M) - \ell(g(n)) = p+p'-q-q' \le 2\gm' q - q - \al
q= 0$ et $M > \dfrac{\gm^2 q^2}{\bt q^2} g(n) = g(n)$, en
contradiction avec \eqref{glandau} et \eqref{lgn}. 
\end{proof}

\begin{definition}\label{defgsuite}
Une $g$--suite de longueur $\ell$ ($1 \le \ell \le +\infty$)
est définie par la donnée  de \dsm{(\gm_k)_{0 \le k \le \ell+1}} 
satisfaisant 
$\gm_0 = 0$, $\gm_1 = \dfrac{1}{2}$ et , pour $1 \le k \le \ell$
\begin{equation}\label{condgamma}
0 < \gm_k < 1
\text{\quad et \quad}
\gm_k < \gm_{k+1}<  \frac{1+\gm_k^2}{2}\cdot
\end{equation}
On lui associe les intervalles 
$I_0 = \intfd{\al_0}{\bt_0} = \intfd{0}{\frac{1}{4}}$, 
$J_0 = \intfd{0}{\frac{1}{2}}$ et,
pour $1 \le k \le \ell$, $I_k$ et $J_k$ définis par
\begin{equation}\label{defIJ}
\al_k = 2\gm_{k+1}-1,\
\bt_k = \gm_k^2,\
I_k = \intfd{\al_k}{\bt_k}
\text{ et }
J_k = \intfd{\gm_k}{\gm_{k+1}}\cdot
\end{equation}
\end{definition}

Il résulte de \eqref{defIJ} que, pour tout $1 \le k \le \ell$, le couple
$(I_k, J_k)$ est un $g$--couple.

Nous étudierons au paragraphe \ref{gsuiteuniforme} 
les \emph{g-suites uniformes}
pour lesquelles le rapport $\al_k/\bt_k$ est constant.

\begin{definition}\label{def5}
Soit une $g$--suite \dsm{(\gm_k)_{0 \le k \le \ell+1}} de longueur
finie $\ell \ge 1$ et $y \ge 12$.
Pour $1 \le k \le \ell$ on note $m_k$ le cardinal de
l'ensemble des indices $j \in \set{0,1,\dots,\ell}$ tels que $I_k \cap
J_j \ne \emptyset$. 

La $g$--suite \dsm{(\gm_k)_{0 \le k \le \ell+1}} est
\emph{$y$-admissible} si pour tout réel $\lb \ge y$ et tout $k$,
$1 \le k \le \ell$, l'intervalle $\lb I_k$ contient au moins $m_k+1$ nombres
premiers, autrement dit, d'après le lemme \ref{lemetak}, si, pour $1 \le k
\le \ell$, on a $\al_k \le \eta_{m_k+1}(\bt_k y) \bt_k$.
\end{definition}

\begin{rem}\label{m1rem}
Puisque $\bt_1 = \gm_1^2 = 1/4$ est contenu dans 
\dsm{J_0 = \intfd{0}{\frac{1}{2}}} on a $m_1 = 1 $.
Pour que la $g$--suite de longueur 1, 
\dsm{\intfd{\al_1}{\frac{1}{4}}, \intfd{\frac{1}{2}}{\frac{1+\al_1}{2}}}
soit $y$--admissible il suffit que pour tout $\lb \ge y$ l'intervalle
\dsm{\intfd{\lb\al_1}{\lb/4}} contienne au moins deux nombres
premiers, c'est à dire que $\al_1 \le \eta_2(y/4)/4$.
En par\-ticulier, pour $\lb=y$, l'intervalle $\intfd{\al_1 y}{y}$
contient au moins $2$ nombres premiers, ce qui nécessite
$y/4 \ge 3$.
\end{rem}

\begin{prop}\label{propg}
Soit $q = \Pplus(g(n))$ le plus grand facteur premier de $g(n)$ et
\dsm{(\gm_k)_{0 \le k \le \ell+1}} une $g$--suite $q$--admissible
de longueur $\ell$.
Alors, pour $0 \le k \le \ell$ l'intervalle
\dsm{qJ_k = q\intfd{\gm_k}{\gm_{k+1}}} contient au
plus un nombre premier qui ne divise pas $g(n)$. De plus
\begin{equation}\label{majqg}
q  \le 
\dfrac{\log g(n)}{\dfrac{\theta(q\gm_{k+1})}{q} -k\dfrac{\log q}{q}
- \dfrac{\sum_{j=1}^{k+1} \log \gm_{j}}{q}} 
\le 
\dfrac{\log g(n)}{\dfrac{\theta(q\gm_{k+1})}{q} -k\dfrac{\log q}{q}
\vphantom{\dfrac{\sum_{j=1}^{k+1} \log \gm_{j}}{q}}}
\cdot
\end{equation}
\end{prop}

\begin{proof}
Soit $\cp(k)$ la propriété \emph{ Il existe au plus un nombre premier
  dans $qJ_k$ qui ne divise pas $g(n)$.}

Pour $k=0$, par le lemme \ref{lemnicolas2}, l'intervalle \dsm{q J_0 =
  \intfd{0}{\frac{q}{2}}} contient au plus un nombre premier qui ne
divise pas $g(n)$ . Ainsi $\cp(0)$ est vrai.

L'intervalle $qI_1 = q\intfd{\al_1}{\frac{1}{4}}$ est contenu dans
$\intfd{0}{q/2}$ qui, par le lemme \ref{lemnicolas2}, contient au plus
un nombre premier ne divisant pas $g(n)$. Compte tenu de la définition
\ref{def5} et de la remarque \ref{m1rem}, $q I_1$ contient au moins
$m_1+1 = 2$ nombres premiers. L'un de ces 2 nombres divise $g(n)$, et
par le lemme \ref{lemgrantham}, il y a au plus un nombre premier dans
$qJ_1$ qui ne divise pas $g(n)$.  Ainsi $\cp(1)$ est vrai.

Supposons $k < \ell$ et $\cp(0)$, $\cp(1)$, \dots,$\cp(k-1)$ vrais.
La borne supérieure de $I_k$ est \dsm{\beta_k = \gm_k^2 < \gm_k.} On a
donc \dsm{ I_k \subset \bigcup_{j=0}^{k-1} J_{j}.  } Par l'hypothèse
de récurrence, chacun des $q J_{j}$ contient au plus un nombre premier
qui ne divise pas $g(n)$. Puisque $q I_k$ rencontre $m_k$ intervalles
$q J_{j}$ il contient au plus $m_k$ nombres premiers qui ne divisent pas
$g(n)$. Or, par définition de la $q$--admissibilité, $q I_k$ contient au
moins $m_k+1$ nombres premiers. L'un d'entre eux divise $g(n)$ et, par
le lemme \ref{lemgrantham}, $J_k$ contient au plus un nombre premier
qui ne divise pas $g(n)$.  C'est--à--dire que $\cp(k)$ est vrai.

On vient de prouver que $g(n)$ est divisible par tous
les nombres premiers de 
$\intfd{0}{q\gm_{k+1}}$ sauf au plus un nombre
premier $q_j \in q\intfd{\gm_{j}}{\gm_{j+1}}$
pour chaque $j=0,1,2,\dots,k$.
Puisque $q$ divise $g(n)$ on a donc
\[
g(n)
\ge\; q\,\frac{\prod_{p\, \le \, q \gm_{k+1}}\, p}{\prod_{j=0}^{k}\, q_j}\;
\ge\; q\,\,\frac{\prod_{p\, \le\, q \gm_{k+1}}\, p}{\prod_{j=0}^{k}\, q\gm_{j+1}}\cdot
\]
On en déduit
\dsm{
\log g(n) \ge \theta(q \gm_{k+1}) 
- \sum_{j=1}^{k+1} \log\gm_{j}
- k\log q.
}
Soit
\[
q\left(
\frac{\theta(q \gm_{k+1})}{q}
- \frac{1}{q}\sum_{j=1}^{k+1} \log\gm_{j}
- k \frac{\log q}{q}
\right)
\le \log g(n).
\]
C'est la première majoration de \eqref{majqg}.
La deuxième résulte de \dsm{\sum_{j=1}^{k+1}\log\gm_j < 0}.
\end{proof}

\begin{prop}\label{prop3}
Soit trois nombres réels positifs $n_0$, $y$, $a$
vérifiant
\[
12 \le y \le a \sqrt{n_0 \log n_0}
\]
et $k \ge 1$ un nombre entier. Faisons l'hypothèse qu'il
existe une $g$--suite $\gm_0,\gm_1,\dots,\gm_{k+1}$, $y$--admissible
de longueur $k$, définissons 
\begin{equation}\label{den}
D_k = \gm_{k+1} \thmin(y\gm_{k+1})-k\frac{\log y}{y}
- \frac{\sum_{j=1}^{k+1}\log \gm_j}{y}
\end{equation}
et supposons $D_k > 0$. Alors, pour $n \ge n_0$, on a
\begin{equation}\label{l54}
\Pplus(g(n)) \le \max(a,b)\sqrt{n\log n}
\qtx{ avec }
b = \frac{1.05314}{D_k}\cdot
\end{equation}
De plus, si $n \ge n_0 \ge 68\,745\,487$, on a
\begin{equation}\label{l54bis}
\Pplus(g(n)) \le \max(a,b')\sqrt{n\log n}
\qtx{avec}
b' = \frac{1}{D_k}\left(1+\frac{\log_2 n_0 - 0.975}{2\log n_0}\right)
\cdot
\end{equation}
\end{prop}

\begin{proof}
Soit $n \ge n_0$ et posons pour simplifier $q = \Pplus(g(n))$.

\noindent
Ou bien $q < y$ et alors
\begin{equation}\label{lab51}
q < y \le a \sqrt{n_0  \log n_0} \le a \sqrt{n \log n}
\end{equation}
ou bien $q \ge y$ et la suite $\gm_0, \gm_1,\dots,\gm_{k+1}$
est a fortiori $q$--admissible. La majoration \eqref{majqg}
de la proposition \ref{propg} donne alors 
\begin{equation}\label{qpetit}
q  \le
\dfrac{\log g(n)}{
\dfrac{\theta(q\gm_{k+1})}{q}
 -k\dfrac{\log q}{q}
- \dfrac{\sum_{j=1}^{k+1} \log \gm_{j}}{q}}.
\end{equation}
On minore \dsm{\theta(q\gm_{k+1})/q} par
\dsm{\gm_{k+1}\thmin(y\gm_{k+1})}
(cf. \S\,  \ref{sectionthmin}).
On remarque aussi que 
\dsm{\sum_{j=1}^{k+1} \abs{\log \gm_{j}} \le (k+1)\log 2  \le 2k\log 2}
et que la fonction 
$t \mapsto (k \log t + \sum_{j=1}^{k+1} \log \gm_j)/t$
est décroissante pour $t \ge 4e$.
La majoration \eqref{qpetit} entraine donc 
$q \le (\log g(n))/D_k$.

On utilise enfin l'inégalité \eqref{L1} qui
donne  la majoration
$q \le b \sqrt{n\log n}$, qui, avec \eqref{lab51}, prouve
\eqref{l54}.
La preuve de \eqref{l54bis} est similaire en majorant
$\log g(n)$ à l'aide de \eqref{L3} au lieu de \eqref{L1}.
\end{proof}

\section{La g-suite $y$--admissible optimale}\label{gsuite2}

La majoration \eqref{majqg} de la proposition \ref{propg} nous conduit
à construire des g-suites $y$--admissibles dont les termes $\gm_k$
soient aussi grands que possibles. C'est l'objet de ce paragraphe.

Soit une $y \ge 12$ et une $g$--suite $y$--admissible
de longueur $1$, $(\gm_0 = 0, \gm_1=1/2, \gm_2)$.
Vu \eqref{defIJ}, $\gm_2 = \dfrac{1+\al_1}{2}$
; la plus grande valeur de $\gm_2$ est donc obtenue en donnant à $\al_1$
la plus grande valeur possible. Or, par la remarque \ref{m1rem}, la
suite $(\gm_0 = 0, \gm_1=1/2, \gm_2)$ est $y$--admissible
si et seulement si $\al_1 \le \frac{1}{4} \eta_2\big(\frac{y}{4}\big)$.
La plus grande valeur de $\gm_2$ est donc obtenue en posant
$\al_1 =  \dfrac{1}{4} \eta_2\big(\dfrac{y}{4}\big)$ et
$\gm_2 = \dfrac{1+\al_1}{2}$.

Soit une $g$--suite $y$--admissible de longueur $k$, $(\gm_j)_{0 \le j
  \le k+1}$ que l'on cherche à prolonger.  La relation \dsm{\gm_{k+1}
  = \bt_{k+1}^2} détermine $\bt_{k+1}$.  La relation \dsm{\gm_{k+2} =
  \frac{1+\al_{k+1}}{2}} montre que la plus grande valeur de
$\gm_{k+2}$ est obtenue en choisissant $\al_{k+1}$ le plus grand
possible.  On pose $m=1$ et on essaie
\begin{equation}\label{choixalpha}
\al_{k+1} 
= \bt_{k+1}\eta_{m+1}(y \bt_{k+1})
\end{equation}
\begin{itemize}
\item
Si $\al_{k+1} \le \al_k$ la construction échoue car
il est impossible de satisfaire
$\gm_{k+1} = \dfrac{1+\al_k}{2}< \dfrac{1+\al_{k+1}}{2} = \gm_{k+2}$.
\item
Si $\al_{k+1} > \al_{k}$  considérons 
\dsm{
I_{k+1} = \intfd{\al_{k+1}}{\bt_{k+1}}
}
Si cet intervalle rencontre au plus $m$ des intervalle
$J_0, J_1, \dots, J_k$ on termine en choisissant 
$\gm_{k+2} = \dfrac{1+\al_{k+1}}{2}$.
Si $I_{k+1}$ rencontre  $m' > m$ intervalles parmi 
$J_0, J_1, \dots, J_{k}$,
il faut recommencer le choix de $\al_{k+1}$ au moyen de la formule
\eqref{choixalpha} en incrémentant $m$.
\end{itemize}

Plus formellement cette construction est décrite dans l'algorithme
\ref{algo1}.  Cet algorithme n'est pas certain de terminer, cependant
il permet de calculer la $g$--suite $(\gm_k)$ de longueur $21$ et
$y$--admissible (avec $y = 4703.39$) qui sera utilisée au paragraphe
\ref{parmajo}.

\floatname{algorithm}{Algorithme}
\renewcommand{\algorithmicrepeat}{\textbf{Répéter}}
\renewcommand{\algorithmicuntil}{\textbf{jusqu'à ce que}}
\renewcommand{\algorithmicendif}{}
\renewcommand{\algorithmicif}{\textbf{si}}
\renewcommand{\algorithmicthen}{\textbf{alors}}
\renewcommand{\algorithmicelse}{\textbf{sinon}}
\begin{algorithm}
\caption{: Calcul de $\al_{k+1}, \bt_{k+1}, \gm_{k+2}$ à partir de $\gm_{k+1}$,
$\al_k$ et $y$}
\label{algo1}
\begin{algorithmic}
\STATE $\bt_{k+1} = \gm_{k+1}^2, \quad m=1$\\
\REPEAT
\STATE $\al_{k+1} = \bt_{k+1} \eta_{m+1}(y \bt_{k+1})$\\
\IF{$\al_{k+1} \le \al_{k}$}
\STATE{Renvoyer ECHEC}
\ELSE
\STATE $m=m+1$\\
\ENDIF
\vspace{-2.5ex}
\UNTIL{\\
  $\quad \intfd{\al_{k+1}}{\bt_{k+1}}$ rencontre au plus $m$
  intervalles $\intfd{0}{\gm_1},
  \dots,\intfd{\gm_{k}}{\gm_{k+1}}$}
\STATE $\gm_{k+2} = \dfrac{1+\al_{k+1}}{2}$\\
\end{algorithmic}
\end{algorithm}

\section{Majoration de  $\log P^{+}(g(n))$ pour $n \ge x$.}\label{parmajo}
\suppress{
Soit $n_0 > 0$ fixé.
On se propose de majorer le rapport
\begin{equation}\label{pab}
\frac{\Pplus(g(n))}{\sqrt{n \log n}} 
\end{equation}
pour $n \ge n_0$ à l'aide le proposition \ref{propg}. 
 Faisons l'hypothèse qu'il existe un $n \ge n_0$ tel que
\begin{equation}\label{pab2}
\Pplus(g(n)) \ge a \sqrt{n \log n}, \qquad a > 0.
\end{equation}

Notons pour simplifier $y=a\sqrt{n_0\log n_0}$, et $q = \Pplus(g(n))$.
On choisit arbitrairement une valeur de $k$ et, par l'algorithme 1, on
construit la $g$--suite $y$--admissible optimale de longueur $k$. Puisque
$q \ge y$ cette suite est a fortiori $q$--admissible et la majoration
\eqref{majqg} de la proposition \ref{propg} donne
}

\begin{theorem}\label{t1}
Pour tout $n \ge 4$ on a
\begin{equation}\label{th1.1}
\frac{\Pplus(g(n))}{\sqrt{n \log n}} 
\le \frac{\Pplus(g(215))}{\sqrt{215 \log(215)}} 
= 1.26542463\dots
\end{equation}
le maximum étant seulement atteint pour $n=215$ avec
$g(215) = 2^3\times 3^2\times 5\times 7\times 11\times 13\times
17\times 19\times 23\times 29\times 31\times 43$
et $\Pplus(g(215)) = 43$.
\end{theorem}

\begin{proof}
On applique la proposition \ref{prop3}
avec $n_0=10^6$, $a = 1.2654$ et 
$y = 4703.39$.
On construit à l'aide de l'algorithme \ref{algo1}
les 21 premiers termes de la $g$--suite $y$--admissible
optimale. 
On obtient les intervalles suivants~: 
\medskip

\begin{center}
\begin{small}
\begin{tabular}{|l|l|l|l|l|l|}
\hline
$k$ & $\al_k$ & $\bt_k$ & $\gm_{k+1}$ & $\set{j}$ & $ D_k$\\
\hline
$1$ & $0.2426\ldots$ & $0.2500\ldots$ & $0.621326\dots$ & $0$   & 0.599249\dots\\ 
$2$ & $0.3786\ldots$ & $0.3860\ldots$ & $0.689343\dots$ & $0$   & 0.663300\ldots\\
$3$ & $0.4669\ldots$ & $0.4751\ldots$ & $0.733450\dots$ &$0$    & 0.706255\ldots\\
$4$ & $0.5308\ldots$ & $0.5379\ldots$ & $0.765402\dots$ &$1$    & 0.739341\ldots\\
$5$ & $0.5780\ldots$ & $0.5858\ldots$ & $0.789031\dots$ &$1$    & 0.760626\ldots\\
$6$ & $0.6135\ldots$ & $0.6225\ldots$ & $0.806763\dots$ &$1, 2$ & 0.776159\ldots\\
$7$ & $0.6422\ldots$ & $0.6508\ldots$ & $0.821112\ldots$ &$2$   & 0.788389\ldots\\
$8$ & $0.6660\ldots$ & $0.6742\ldots$ & $0.833025\dots$ &$2$    & 0.798242\ldots\\
$9$ & $0.6845\ldots$ & $0.6939\ldots$ & $0.842282\dots$ &$2, 3$ & 0.805505\ldots\\
$10$& $0.7019\ldots$ & $0.7094\ldots$ & $0.850985\dots$ &$3$    & 0.812224\ldots\\
$11$& $0.7165\ldots$ & $0.7241\ldots$ & $0.858275\dots$ &$3$    & 0.817565\ldots\\ 
$12$& $0.7266\ldots$ & $0.7366\ldots$ & $0.863347\dots$ &$3, 4$ & 0.820742\ldots\\
$13$& $0.7375\ldots$ & $0.7453\ldots$ & $0.868760\dots$ &$4$    & 0.824250\ldots\\
$14$& $0.7467\ldots$ & $0.7547\ldots$ & $0.873399\dots$ &$4$    & 0.827003\ldots\\
$15$& $0.7547\ldots$ & $0.7628\ldots$ & $0.877397\dots$ &$4$    & 0.829130\ldots\\
$16$& $0.7594\ldots$ & $0.7698\ldots$ & $0.879717\dots$ &$4, 5$ & 0.829621\ldots\\
$17$& $0.7657\ldots$ & $0.7739\ldots$ & $0.882877\dots$ &$5$    & 0.830930\ldots\\
$18$& $0.7712\ldots$ & $0.7794\ldots$ & $0.885632\dots$ &$5$    & 0.831844\ldots\\
$19$& $0.7760\ldots$ & $0.7843\ldots$ & $0.888043\dots$ &$5$    & 0.832421\ldots\\
$20$& $0.7803\ldots$ & $0.7886\ldots$ & $0.890159\dots$ &$5$    & 0.832710\ldots\\
$21$& $0.7816\ldots$ & $0.7923\ldots$ & $0.890844\dots$ &$5, 6$ & 0.831605\ldots\\
\hline
\end{tabular}
\end{small}
\end{center}
La colonne $\set{j}$ contient les valeurs de $j$ telles
que $I_k$ rencontre $J_j$.
Ainsi la valeur de $m_k$ est le nombre des valeurs figurant dans
la $k\iem$ ligne de la colonne $\set{j}$.
\medskip

Cela donne $D_{20} =  0.832710\dots$ et, avec \eqref{l54}, 
$b = 1.264713\dots < a$ ce qui prouve que
$\Pplus(g(n)) < 1.2654 \sqrt{n \log n}$ pour $n \ge 10^6$.
Le calcul de toutes les valeurs de $g(n)$ pour $4 \le n \le 1\,000\,000$
montre que le maximum est atteint une seule fois, en $n  = 215$.
\end{proof}

\begin{rem}
Pour $y = 4703.39$, l'algorithme 1 calcule
$\al_k$, $\bt_k$ et $\gm_{k+1}$ pour $k \le 30$, mais trouve
$\al_{31} < \al_{30}$ et retourne donc \og ECHEC \fg.
Les valeurs de $D_k$ calculées par la formule \eqref{den} vérifient
$D_{20} > D_{21} > \cdots > D_{30} = 0.822869\dots$

Les valeurs de $\al_k$, $\bt_k$, $\gm_{k+1}$ déterminées
par l'algorithme 1 ne dépendent que de façon discrète
de $y$. Par exemple, on obtient la même $g$--suite
de longueur $21$ pour tout $y$ vérifiant
$4692 \le y \le 4859$. Cependant $D_k$ dépend de $y$.
Notons aussi que $\al_k$, $\bt_k$ et $\gm_k$
sont rationnels, mais avec des numérateurs et dénominateurs
croissant très vite avec $k$.

Dans la preuve du théorème de \cite{GRA}, les suites
$\al_1,\dots,\al_9$,
$\bt_1,\dots,\bt_9$ utilisées par J. Grantham sont très
voisines de celles obtenues par l'algorithme 1 pour
$y = 3329$.

Soit $y = 114\,620$. En calculant avec l'algorithme 1 la $g$--suite
optimale $y$--admissible de de longueur $97$, on trouve $\gm_{98} =
0.9693673\dots, D_{97}=0.9549879\dots$. Avec $n_0 = 540\,000\,000$ et
$a=1.1$, la formule \eqref{l54bis} de la proposition \ref{prop3} donne
alors $b' = 1.0998903\dots$ et
\begin{equation}\label{Pplus1.1}
\Pplus(g(n)) \le 1.1\sqrt{n\log n},
\qquad n \ge 540\,000\,000.
\end{equation}

\end{rem}

\section{La $g$--suite uniforme}\label{gsuiteuniforme}

L'étude théorique des $g$--suites optimales ne semble pas
facile. Dans ce paragraphe nous introduisons les $g$--suites
uniformes, moins efficaces pour les calculs numériques, mais
plus simples à étudier.

\begin{definition}\label{suitecanonique}
Soit $0 < \eta < 1$. On pose $\gm_0 = 0$ et, pour $j \ge 1$,
on définit $\gm_j = \gm_j(\eta)$ par
\begin{equation}
\gm_j = \frac{1+\eta\gm_{j-1}^2}{2}\cdot
\end{equation}
\end{definition}

\begin{rem}\label{gmcroit}
Remarquons que $\gm_j(\eta)$ est une fonction croissante de $j$ et
de $\eta$.
\end{rem}

\begin{lem}\label{lemeps}
  La suite $\gm_j(\eta)$ définie ci-dessus est une
  $g$--suite infinie.  On l'appelle la \emph{$g$--suite uniforme de
    paramètre} $\eta$.
  Notons $\veps = \veps(\eta) = 1-\eta$, 
et $L_{\veps} = \lim_{j \to +\infty} \gm_j$.
  Alors
\begin{equation}\label{gmlim}
L_{\veps} = \frac{1}{1+\sqrt\veps}
\qtx{ et, pour tout $j \ge 0$, }
L_{\veps} - \gm_j \le L_{\veps}(1-\sqrt\veps)^j.
\end{equation}
\end{lem}

\begin{proof}
La démonstration se fait par récurrence.
On a $\gm_0 = 0$, $\gm_1 = \dfrac{1+\eta\gm_0^2}{2} = \dfrac12$,
puis 
\[
\gm_{j+1} = \frac{1+\eta\gm_j^2}{2} < \frac{1+\gm_j^2}{2}\cdot
\]
En outre $\gm_{j+1} = f(\gm_j)$ avec 
\dsm{f = t \mapsto [1+(1-\veps)t^2]/2}. La fonction $f$ est croissante
pour $t \ge 0$
et admet deux points fixes qui sont \dsm{\frac{1}{1+\sqrt\veps}} 
et \dsm{\frac{1}{1-\sqrt\veps}}.
Puisque \dsm{\gm_0 < \gm_1}
la suite
$(\gm_j)$ est strictement croissante de limite
$L_{\veps}=\dfrac{1}{1+\sqrt\veps}$.
Les conditions figurant dans la définition \ref{defgsuite} sont
donc satisfaites et $(\gm_j)$ est une $g$--suite. 
De plus, puisque $f'$ est croissante
\begin{eqnarray*}
L_{\veps} - \gm_{j} &=& f(L_{\veps}) - f(\gm_{j-1}) 
< f'(L_{\veps})  (L_{\veps} - \gm_{j-1})
 = (1-\sqrt\veps)(L_{\veps} -\gm_{j-1})\\
&\le& (1-\sqrt\veps)^j(L_{\veps} -\gm_{0}) = L_{\veps}(1-\sqrt\veps)^j.
\end{eqnarray*}
\end{proof}

Dans tout ce paragraphe  $\eta$ est un réel positif
satisfaisant $0 < \eta < 1$, $\veps = 1-\eta$,
$(\gm_j)$ est la $g$--suite uniforme de
paramètre $\eta$, et $I_j = \intfd{\al_j}{\bt_j}$,
$J_j = \intfd{\gm_j}{\gm_{j+1}}$ sont les intervalles
associés à cette $g$--suite (cf. défintion \ref{defgsuite}).

\begin{lem}\label{lemgn1n}
Soit $u_j = L_{\veps}-\gm_j$. Alors $(\gm_{j+1}-\gm_j)$ est une
suite décroissante, et, pour tout $n$ on a
\[
\gm_{j+1} - \gm_j = \sqrt\veps u_j + \frac{u_j^2}{2}(1-\veps).
\]
En particulier 
\dsm{\sqrt\veps (L_{\veps}-\gm_j) < \gm_{j+1} - \gm_j}.
\end{lem}

\begin{proof}
En effet, puisque $\gm_j = L_{\veps}-u_j =
\dfrac{1}{1+\sqrt\veps}-u_j$,
on a
\begin{eqnarray*}
2(\gm_{j+1} - \gm_j) &=& 2f(\gm_j) - 2\gm_j
= 1 + (1-\veps)\gm_j^2 - 2\gm_j\\
&=&
 1+ (1-\veps)\left(\frac{1}{1+\sqrt\veps}-u_j\right)^2 
-\frac{2}{1+\sqrt\veps}+2u_j\\
&=&
 2\sqrt\veps u_j + (1-\veps)u_j^2
> 2\sqrt{\veps}u_j = 2\sqrt{\veps}(L_{\veps}-\gm_j).
\end{eqnarray*}
\end{proof}

\begin{lem}\label{lem2g}
Il n'existe pas de couple $(k,j)$ tel que
\dsm{\intfd{\gm_j}{\gm_{j+1}} \subset \intfd{\al_k}{\bt_{k}}}.
Et donc chaque intervalle $I_k$ rencontre au plus deux
intervalles $J_j$.
\end{lem}

\begin{proof} Supposons
\dsm{\intfd{\gm_j}{\gm_{j+1}} \subset \intfd{\al_k}{\bt_k} }.
Alors
\[
\al_k \le \gm_j < \gm_{j+1} \le \bt_k = \gm_k^2
\]
et donc
\begin{equation}\label{lgmn}
 L_{\veps}-\gm_j > L_{\veps}-\gm_k^2 > L_{\veps}-L_{\veps}^2 
= L_{\veps}(1-L_{\veps}) =  \sqrt\veps L_{\veps}^2.
\end{equation}
En utilisant le lemme \ref{lemgn1n}, on a aussi
\[
\sqrt\veps(L_{\veps}-\gm_j) < \gm_{j+1}-\gm_j  \le \bt_k-\al_k = \veps\bt_k
= \veps\gm_k^2 < \veps L_{\veps}^2
\]
qui implique \dsm{L_{\veps}-\gm_j < \sqrt\veps  L_{\veps}^2}, 
ce qui contredit \eqref{lgmn}.
\end{proof}

\begin{lem}\label{lem12}
Soit $q$ un facteur premier de $g(n)$ et
$\eta \le \eta_3(q/4)$ ;
la $g$--suite uniforme de paramètre $\eta$ et
$\big(\intfg{\al_k}{\bt_k}, \intfg{\gm_k}{\gm_{k+1}}\big)$
est $q$--admissible.
Pour tout $k \ge 1$,
\begin{equation}\label{eq11}
q  \le 
\dfrac{\log g(n)}{\dfrac{\theta(q\gm_{k+1})}{q} -k\dfrac{\log q}{q}} 
\cdot
\end{equation}
\end{lem}

\begin{proof}
  Par le lemme \ref{lem2g} chaque intervalle $I_k$ rencontre au plus 2
  des intervalles $J_j$. Le nombre $m_k$ introduit dans la définition
  \ref{def5} vérifie donc $m_k \le 2$ pour tout $k$. L'inégalité
  $\bt_1 = \dfrac{1}{4} \le \bt_k$ et la croissance de $\eta_3$
  donnent $\al_k \le \bt_k \eta_3\big(\dfrac{q}{4}\big) \le \bt_k
  \eta_3(q \bt_k)$.  On a donc $q\al_k \le q\bt_k\eta_3(q\bt_k)$, et
  par le lemme \ref{lemetak} l'intervalle $qI_k$ contient au moins 3
  nombres premiers.  Ainsi, par la définion \ref{def5},
  la $g$--suite uniforme est $q$--admissible, la
  propostion \ref{propg} s'applique, ce qui termine la preuve.
\end{proof}

\begin{lem}\label{qmaj}
Soit $n$ entier et  $q$ le plus grand facteur premier de $g(n)$.
Lorsque $n$ tend vers l'infini, on a
\begin{equation}
q = \Pplus(g(n)) \le \log g(n) (1+ \bigo{\veps})
\end{equation}
avec $\veps = \veps_1 + \sqrt\veps_2$ o\`u $\veps_1$ et $\veps_2$ 
sont définis par
\[
\veps_1 = \max_{x \ge \frac{q}{2}}\abs{\frac{\theta(x)}{x}-1}
\qquad 
\veps_2 = 
\max\left(
1-\eta_3\Big(\frac{q}{4}\Big),
\pfrac{\log q}{\sqrt q}^2
\right)
< 1.
\]
\end{lem}

\begin{proof}
Soit
\begin{equation}\label{defk}
\eta = 1-\veps_2
\qtx{ et }
k=\intfrac{\log q}{\sqrt\veps_2}\cdot
\end{equation}
Lorsque $n \to +\infty$, par \eqref{lab5}, 
$q = \Pplus(g(n))$ tend aussi vers l'infini, $\veps_1$, $\veps_2$
et $\veps$ tendent vers $0$ et $k$ tend vers l'infini.

Puisque \dsm{\eta \le \eta_3\left(\frac{q}{4}\right)},
le lemme \ref{lem12} s'applique, et l'on a
$q\gm_{k+1} \ge q\gm_1 = q/2$.
Par définition de $\veps_1$, il vient
\dsm{
\frac{\theta(q\gm_{k+1})}{q\gm_{k+1}} \ge 1 - \veps_1,
}
d'où
\begin{eqnarray}\nonumber
\frac{\theta(\gm_ {k+1} q)}{q}  &\ge& 
\gm_{k+1} - \gm_{k+1} \veps_1\\
&\ge& \gm_{k+1} - \veps_1
= 1 - \veps_1 - (1-L_{\veps_2}) - (L_{\veps_2}-\gm_{k+1}).
\label{l121}
\end{eqnarray}
Une première utilisation du lemme \ref{lemeps} donne
\begin{equation}\label{l122}
1-L_{\veps_2} = 1 - \frac{1}{1+\sqrt\veps_2} \le \sqrt\veps_2.
\end{equation}
Par \eqref{defk}, on a  $k+1 > \dfrac{\log q}{\sqrt\veps_2}$;
une deuxième utilisation du lemme \ref{lemeps} donne
\begin{equation*}
L_{\veps_2} - \gm_{k+1} \le
(1-\sqrt\veps_2)^{k+1} \le (1-\sqrt\veps_2)^{\frac{\log q}{\sqrt\veps_2}}
\le \pfrac{1}{e}^{\log q}
= \frac{1}{q}\cdot
\end{equation*}
Avec \eqref{l121}, \eqref{l122} et la défintion de $\veps$ cela donne
\begin{equation}\label{qgkp1q}
\frac{\theta(q\gm_{k+1})}{q} \ge
  1 - \veps_1 - \sqrt\veps_2 - \frac{1}{q}
=  1 - \veps - \frac{1}{q}\cdot
\end{equation}
La définition \eqref{defk} de $k$ donne
\dsm{
k \frac{\log q}{q} \le
\dfrac{\log q}{\sqrt\veps_2}\dfrac{\log q}{q} 
\cdot
}
Avec \eqref{qgkp1q}, on en déduit
\[
\frac{\theta(q\gm_{k+1})}{q} - k \frac{\log q}{q} \ge
1-\veps - \dfrac{1}{q} - \frac{\log^2 q}{q\sqrt\veps_2}\cdot
\]
Par définition de $\veps_2$, pour $q \ge 3$, on a 
\dsm{\frac{1}{q} \le \frac{\log^2 q}{q} \le \veps_2 \le \sqrt{\veps_2}},
ce qui donne 
\[
\frac{\theta(q\gm_{k+1})}{q} - k \frac{\log q}{q} \ge
1 - \veps - \frac{1}{q} - \sqrt{\veps_2}  \ge 1 - \veps -2\sqrt\veps_2
\ge 1 - 3\veps
\]
et termine la preuve avec \eqref{eq11}.
\end{proof}

\section{Majoration asymptotique de $\Pplus(g(n))$}

\begin{theorem}\label{t2}
Soit $P^{+}(g(n))$ le plus grand facteur premier de $g(n)$.
Lorsque $n$ tend vers l'infini, $P^{+}(g(n))$ est majoré par~:
\begin{enumerate}
\item
Sans aucune hypothèse, il existe $a > 0$ tel que
\begin{equation}\label{truemaj}
P^{+}(g(n)) \le \sqrt{\Li^{-1}(n)} + \bigo{\sqrt n e^{-a\sqrt{\log n}}}.
\end{equation}
\item
Si l'hypothèse de Riemann est vraie
\begin{equation}\label{Riemanmaj}
P^{+}(g(n)) \le \sqrt{\Li^{-1}(n)} + \bigo{n^{3/8}(\log n)^{7/8}}.
\end{equation}
\item
Si l'hypothèse de Riemann et la conjecture de Cramér \eqref{conjcra}
sont vraies,
\begin{equation}\label{Pplusjcra}
\Pplus(g(n)) \le \sqrt{\Li^{-1}(n)} + \bigo{n^{1/4}(\log n)^{9/4}}. 
\end{equation}
\end{enumerate}
\end{theorem}

\begin{proof}
Soit $q=\Pplus(g(n))$. Par \eqref{lab5} on a
\begin{equation}\label{qequiv}
q \sim \sqrt{n \log n}
\qtx{ et }
\log q  \sim \frac{1}{2} \log n.
\end{equation}
Nous utiliserons aussi l'équivalence
\begin{equation}\label{L1equiv}
\Li^{-1}(n) \sim n \log n.
\end{equation}
Nous allons appliquer le lemme \ref{qmaj} en évaluant dans les  
3 cas les quantités $\veps_1$, $\veps_2$ et 
$\veps = \veps_1 + \sqrt{\veps_2}$.

\begin{enumerate}
\item
Par le théorème des nombres premiers, \eqref{TNP1}, il existe 
$a_1 > 0$ tel que
\[
\veps_1 = \bigo{\exp(-a_1\sqrt{\log n})}.
\]
Par la proposition \ref{propmin}, 1., on a
\[
1-\eta_3\left(\frac{q}{4}\right) = \bigo{q^{-0.475}}.
\]
On a donc, par \eqref{qequiv},
\begin{equation}\label{}
\sqrt\veps_2 = \bigo{q^{-0.2375}} = \bigo{n^{-0.11875}}
\end{equation}
et $\veps = \veps_1 + \sqrt{\veps_2} =\bigo{\exp(-a_1\sqrt{\log n}}$.
L'application du lemme \ref{qmaj} fournit ensuite l'inégalité
\begin{equation}\label{majq1}
q \le (\log g(n)) \Big(1+ \bigo{\exp(-a_1\sqrt{\log n})}
\Big).
\end{equation}
Par \eqref{asy} il existe $a_2 > 0$ tel que
\[
\log g(n) = \sqrt{\Li^{-1}(n)} + 
\bigo{\sqrt ne^{-a_2\sqrt{\log n}}}
\]
ce qui, avec \eqref{majq1} et \eqref{L1equiv} démontre \eqref{truemaj}
pour $a < \min(a_1,a_2)$.

\item
Si l'hypothèse de Riemann est vraie, on a par \eqref{TNP2},
\begin{equation}\label{majeps1R}
\veps_1 = \bigo{\frac{\log^2 q}{\sqrt q}}
\end{equation}
et le point 3. de la proposition \ref{propmin} donne
\[
1-\eta_3\pfrac{q}{4} = \bigo{\frac{\log q}{\sqrt q}}\cdot
\]
On obtient donc
\[
\veps_2 = \bigo{\log q/\sqrt q} 
\]
et, par \eqref{qequiv}
\[
\veps = \veps_1 + \sqrt{\veps_2} = \bigo{\frac{\sqrt{\log q}}{q^{1/4}}}
= \bigo{\frac{(\log n)^{3/8}}{n^{1/8}}}
\]
puis, par le lemme \ref{qmaj},
\begin{equation*}
q  \le \log g(n)
\left(1 + \bigo{\frac{(\log n)^{3/8}}{n^{1/8}}}\right)
\end{equation*}
qui, avec \eqref{asyTH} et \eqref{L1equiv} donne 
\eqref{Riemanmaj}.

\item
L'estimation \eqref{majeps1R} reste valable, tandis que \eqref{etacra}
donne $\veps_2 = \bigo{\dfrac{\log^2 q}{q}}$  et donc
\[
\veps = \bigo{\frac{\log^2 q}{\sqrt q}} 
= \bigo{\frac{(\log n)^{7/4}}{n^{1/4}}},
\]
qui, par le lemme \ref{qmaj}, \eqref{asyTH}
et \eqref{L1equiv}, prouve \eqref{Pplusjcra}.

On notera que \eqref{Pplusjcra} reste valable si l'on remplace la
conjecture de Cramér \eqref{conjcra} par la conjecture plus faible
\dsm{
p_{i+1} - p_i = \bigo{\log^4 p_i}.
}
\end{enumerate}
\end{proof}

\section{Majoration effective de $\Pplus(g(n))$}
\begin{theorem}\label{t3}
Pour $n \ge 2$ on a
\begin{equation}\label{eqt3}
\Pplus(g(n)) \le \log g(n) \bigg(1+\frac{5.54}{\log n}\bigg)
\le \sqrt{n\log n}\left(1 + \frac{\log_2 n+10.8}{2\log n}\right).
\end{equation}
\end{theorem}

\begin{proof}
Pour $2 \le n \le 10^6$, on calcule $\Pplus(g(n))$ et $\log(g(n))$
par l'algorithme naïf de (\cite{DNZ}, \cite{NICRIRO})
et on vérifie que les inégalités \eqref{eqt3} sont satisfaites.
Soit $n \ge 10^6$. Supposons que l'on ait
\begin{equation}\label{qsqrtn}
q > \log g(n)\left(1+ \frac{5}{\log n}\right).
\end{equation}
Puisque $n \ge 10^6 > 906$ il résulte de \eqref{L2} que
\[
q > \sqrt{n \log n}\left(1+ \frac{5}{\log n}\right) =
\sqrt{n} \left(\sqrt{\log n} + \frac{5}{\sqrt{\log n}} \right).
\]
Avec la croissance de $t \mapsto t + \dfrac{5}{t}$ pour $t \ge \sqrt
5$ et $\sqrt{\log{10^6}} \approx 3.7 > \sqrt 5$, on en déduit \dsm{q >
  \sqrt{n}\left(\sqrt{\log 10^6}+ \frac{5}{\sqrt{\log 10^6}} \right)
  > 5.062 \sqrt{n} \ge 5\,062}. Le plus petit nombre premier $> 5\,
062$ est $5\,077$. On a donc

\begin{equation}\label{qnmin}
q \ge 5\,077,\qquad
0.25\, q \ge 1269
\end{equation}
et
\begin{equation}\label{log25q}
\log(0.25\, q) > \log (0.25 \times 5.062 \sqrt{n}) 
> \log \sqrt{n}) 
\ge 0.5 \log n.
\end{equation}
Définissons $\eta$ et $\veps$ par 
\begin{equation}\label{defeta}
\veps = \frac{7.02}{\log^2 n}
\qtx{ et }
\eta = 1-\veps.
\end{equation}
De \eqref{defeta} on déduit d'une part
\begin{equation}\label{minsqrteps}
\frac{2.649}{\log n} \le \sqrt{\veps} \le \frac{2.650}{\log n}
\end{equation}
et, aussi,  avec $n \ge 10^6$, 
\begin{equation}\label{minunmoinseps}
\veps < 0.0368
\qtx{ et }
\eta = 1-\veps > 0.9632.
\end{equation}
La définiton \eqref{defeta} de $\veps$,
la minoration \eqref{log25q}, la table  \ref{tabledelta3} et 
enfin la minoration \eqref{qnmin} avec la décroissance
de $\dt_3$ donnent
\begin{equation}\label{majeps}
 \veps = \frac{7.02}{\log^2 n} \ge 
\frac{1.755}{\left(\log(0.25\,q)\right)^2} 
\ge \frac{\dt_3(1\,269)}{\left(\log(0.25\,q)\right)^2} 
\ge \frac{\dt_3(0.25\, q)}{\left(\log(0.25\,q)\right)^2}
\end{equation}
et donc, avec la majoration \eqref{pd3}, (dans laquelle on pose $x=y=0.25\,q$),
\[
\eta = 1 - \veps 
\le 1 - \frac{\dt_3(0.25\,  q)}{\left(\log(0.25\,q)\right)^2}
\le \eta_3\left(0.25\,q \right).
\]
La $g$--suite uniforme de paramètre $\eta$
satisfait les hypothèses du lemme \ref{lem12}.
On applique ce lemme en choisissant
\begin{equation}\label{defkeff}
k = \lfloor 0.64\,\log n \log_2 n \rfloor
\ge\lfloor 0.64\,\log 10^6 \log_2 10^6 \rfloor
= 23.
\end{equation}
Nous avons remarqué au paragraphe \ref{gsuiteuniforme}
(remarque \ref{gmcroit})
que $\gm_k = \gm_k(\eta)$ est une fonction croissante
de $\eta$ et de $k$.
Puisque, par \eqref{minunmoinseps}, $\eta > 0.9632$
et $k \ge 23$ on a  $\gm_{k+1} \ge \gm_{24}(0.9632) > 0.8378$.
Par \eqref{qnmin} et \eqref{log25q}, on en déduit
\[
q\gm_{k+1} \ge 0.8378 \cdot 5\,077  \ge 4253
\qtx{ et }
\log q\gm_{k+1} \ge \log 0.25\, q \ge 0.5 \log n
\]
puis, par \eqref{propthetad}, et à l'aide de la table
\ref{tableThetad}
\begin{equation}\label{thetaqq}
\frac{\theta(q\gm_{k+1})}{q\gm_{k+1}} \ge  
1 - \frac{\theta_d(4253)}{(\log q\gm_{k+1})^2}
\ge
1 - \frac{1.863}{(\log q\gm_{k+1})^2}
\ge 1 - \frac{7.452}{(\log n)^2}\cdot
\end{equation}
En utilisant  $n \ge 10^6$, on en déduit
\begin{equation}\label{minqgmq}
\frac{\theta(q\gm_{k+1})}{q} \ge 
\left(1 - \frac{7.452}{\log^2 n}\right)\gm_{k+1}
\ge \left(1-\frac{0.54}{\log n} \right)\gm_{k+1}.
\end{equation}
Minorons enfin $\gm_{k+1}$~: De \eqref{gmlim}, en utilisant
\eqref{minsqrteps} et \eqref{defkeff}, on déduit
\begin{eqnarray*}
\gm_{k+1} &\ge& \frac{1}{1+\sqrt\veps} 
- \left(1-\sqrt{\veps}\right)^{k+1} 
\ge 
 1 - \sqrt{\veps} - \left(1-\frac{2.649}{\log n} \right)^{k+1}\\
&\ge& 1-\sqrt{\veps} - 
\left(1-\frac{2.649}{\log n} \right)^{0.64\, \log_2 n \log n}
 \\
&\ge& 1 - \sqrt{\veps} - 
\left[\left(1-\frac{2.649}{\log  n}\right)^{\frac{\log n}{2.649}}%
\right]^{0.64\,\times 2.649\, \log_2(n)}\\
&\ge& 1 - \frac{2.65}{\log n} - \frac{1}{(\log n)^{0.64\times2.649}} \ge 1-\frac{2.812}{\log n}\cdot
\end{eqnarray*}
Avec \eqref{minqgmq} cela donne
\begin{equation}\label{eq1f}
\frac{\theta(q\gm_{k+1})}{q} \ge
\left(1-\frac{0.54}{\log n} \right)
\left( 1-\frac{2.812}{\log n} \right)
\ge 1 - \frac{3.352}{\log n} + \frac{1.518}{\log^2 n}\cdot
\end{equation}
Puisque \dsm{\frac{\log t}{t}} est une fonction décroissante de $t$
pour $t \ge e$, et $q \ge \sqrt{n\log n}$, on a
\[
\frac{\log q}{q} \le \frac{1}{2}\frac{\log(n\log n)}{\sqrt{n\log n}}\cdot
\]
En utilisant \eqref{defkeff} 
et la décroissance de
\dsm{t \mapsto (\log t)^{3/2} \log_2 t \log(t \log t)/\sqrt t},
pour $t \ge 10^6$, on en déduit
\begin{equation}\label{eq2f}
k\frac{\log q}{q} \le 0.32\,\log n \log_2 n
\frac{\log(n\log n)}{\sqrt{n\log n}} 
\le 
 \frac{0.71}{\log n}\cdot
\end{equation}
Avec la formule \eqref{eq11} du lemme \ref{lem12},
\eqref{eq1f} et \eqref{eq2f} donnent
\[
q \le \frac{\log g(n)}{1-\dfrac{4.062}{\log n}+\dfrac{1.518}{\log^2 n}}\cdot
\]
Or, sur l'intervalle $0 \le X \le 1/\log 10^6$, la fraction
\dsm{\frac{4.062 - 1.518 X}{1-4.062 X + 1.518 X^2}} 
est croissante et majorée par $5.54$. On en déduit,
pour $n \ge 10^6$, 
\[
\frac{1}{1-\dfrac{4.062}{\log n}+\dfrac{1.518}{\log^2 n}}
= 1 + \frac{\dfrac{4.062}{\log n}-\dfrac{1.518}{\log^2 n}}{1-\dfrac{4.062}{\log n}+\dfrac{1.518}{\log^2 n}}
\le 1+\dfrac{5.54}{\log n}
\]
ce qui, avec \eqref{qsqrtn} prouve la première inégalité
de \eqref{eqt3}. 
En appliquant \eqref{L3}, on en déduit
\begin{eqnarray*}
q &\le&  \sqrt{n\log n} 
\left(1+\frac{\log_2 n-0.975}{2\log n} \right)%
\left(1+\dfrac{5.54}{\log n}\right)
\\
&\le&
\sqrt{n\log n} 
\left(1+\frac{\log_2 n+10.11}{2\log n} 
+ \dfrac{5.54}{\log n}\frac{\log_2 n-0.975}{2\log n} \right)\\
&\le&
\sqrt{n\log n} 
\left(1+\frac{\log_2 n+10.11}{2\log n} 
+ \dfrac{1}{2\log n}\dfrac{5.54(\log_2 10^6-0.975)}{\log 10^6} \right)
\end{eqnarray*}
ce qui complète la preuve de \eqref{eqt3}.
\end{proof}

\section{Minoration de $\Pplus(g(n))$}\label{parmino}

\begin{lem}\label{parminolemm1}
Soit $n \ge 2$, $\al \ge 2$ et $\lb$ premier tel que $\lb^{\al}$
divise $g(n)$. Soit $q$ le nombre premier suivant $\Pplus(g(n))$.
On a
\begin{equation}\label{m11}
\lb^{\al} - \lb^{\al-1} < q,
\end{equation}
\begin{equation}\label{m12}
\lb < q^{1/\al} + 1 \cdot
\end{equation}
et
\begin{equation}\label{m13}
\al < \frac{\log(2q)}{\log 2}\cdot
\end{equation}
\end{lem}

\begin{proof}
\eqref{m11} et \eqref{m12} constituent la propriété 5 de
\cite{NICAA}, dont la preuve est facile.  De \eqref{m11} on déduit
\[
\lb^{\al} < \frac{q}{1-\frac{1}{\lb}} \le 2q
\]
ce qui implique
\dsm{\al < \log(2q)/\log(\lb) \le \log(2q)/\log(2)}
et prouve \eqref{m13}. 
\end{proof}

\begin{lem}\label{lemm2}
  Soit $\theta$ la fonction de Chebyshev définie en
  \eqref{defthetapsi}. Pour tout $n \ge 2$ on a la majoration
\begin{equation}\label{m21}
\log g(n) \le \theta\left(\Pplus(g(n)\right) + S(q)
\end{equation}
o\`u $q$ est le nombre premier suivant $\Pplus(g(n))$ et
\begin{equation}\label{m22}
S(q) = \sum_{\al=2}^{\log(2q)/\log 2} \theta(q^{1/\al}+1).
\end{equation}
\end{lem}

\begin{proof}
En désignant par $\lb$ un nombre premier générique et en
appliquant le lemme \ref{parminolemm1} il vient
\begin{eqnarray}
\log g(n) &\le& \sum_{\lb \le \Pplus(g(n))}\log\lb 
+ \sum_{\al \ge 2}\; \sum_{\lb^{\al} \dv g(n)} \log \lb \\
&\le& \theta\Big(\Pplus(g(n))\Big) 
+ \sum_{\al=2}^{\log(2q)/\log 2} \sum_{\lb\le q^{1/\al}+1} \log\lb\\
&=& \theta\Big(\Pplus(g(n))\Big) + S(q).
\end{eqnarray}
\end{proof}

\begin{lem}\label{lemm3}
Soit $q \ge 3$ un nombre premier, $p$ le nombre premier
précédant $q$ et $S(q)$ la somme définie par \eqref{m22}.
On a
\begin{equation}\label{m31}
S(q) < 2.13\, \sqrt p.
\end{equation}
\end{lem}

\begin{proof}
Supposons d'abord $q \ge q_0 = 100\, 000$.
Il résulte de la définition de $\eta_1$ (cf. Définition \eqref{defetak})
et de la table \ref{tableeta1} que
\begin{equation}\label{m32}
\frac{p}{q} \ge \eta_1(q_0) = \frac{107\,377}{107\,441} \ge 0.9994.
\end{equation}
En posant \dsm{k= \intfrac{\log(2q)}{\log 2}}, en utilisant
l'inégalité \eqref{T2} et en remarquant que 
$t \mapsto \dfrac{\log t}{\sqrt t}$ et
$t \mapsto \dfrac{\log(t/4)}{t^{1/4}}$ sont décroissantes pour 
$t \ge e^2 \approx 7.39$ et $t \ge 4e^4 \approx 218.4$ 
respectivement, il vient

\begin{eqnarray*}
\frac{S(q)}{\sqrt p} &\le&  1.000028\,
 \frac{\sqrt q}{\sqrt p} \sum_{\al=2}^{k} \frac{q^{1/\al}+1}{\sqrt
   q}\\
&\le& \frac{1.000028}{\sqrt{\eta_1(q_0)}}
\left(\frac{k-1}{\sqrt q} + 1 + \frac{1}{q^{1/6}} + \frac{k-3}{q^{1/4}} \right)\\
&\le&
\frac{1.000028}{\sqrt{\eta_1(q_0)}}
\left(\frac{\log q}{\sqrt q \log 2} + 1 + \frac{1}{q^{1/6}} 
+ \frac{\log(0.25\,q)}{q^{1/4}\log 2} \right)\\
&\le&
\frac{1.000028}{\sqrt{\eta_1(q_0)}}
\left(\frac{\log q_0}{\sqrt q_0 \log 2} + 1 + \frac{1}{q_0^{1/6}} 
+ \frac{\log(q_0/4)}{q_0^{1/4}\log 2} \right)
= 2.0215\dots,
\end{eqnarray*}
ce qui prouve \eqref{m31} pour $q \ge 100\,000$. 
Ensuite on calcule $S(q)/\sqrt{p}$ pour tous les $q$ premiers, $3 \le
q \le 100\,000$, et l'on trouve que le maximum est atteint pour $q=17$
et vaut \dsm{\frac{\log 2160}{\sqrt{13}} = 2.129\dots}
\end{proof}

\begin{lem}\label{lemm4}
Pour tout $n \ge 906$ on a
\begin{equation}\label{m41}
\Pplus(g(n))  
 \ge \frac{\sqrt{n \log n}}{1.000028} - 2.4\, (n \log n)^{1/4}\cdot
\end{equation}
\end{lem}

\begin{proof}
Par \eqref{m21} et \eqref{m31} il vient, en posant $P =
\Pplus(g(n))$
\[
\log  g(n) \le \theta(P) + 2.13 \sqrt{P}.
\]
Or, vu \eqref{L2}, pour $n \ge 906$, 
$\log g(n) \ge \sqrt{n \log n}$ et l'on a la majoration
\eqref{th1.1}.
Avec l'inégalité \eqref{T2} cela entraine
\[
\sqrt{n \log n} \le \theta(P) + 2.13 \sqrt{P}
\le 1.000028 P + 2.13 \times \sqrt{1.266} (n \log n)^{1/4},
\]
d'o\`u\eqref{m41}.
\end{proof}

\begin{theorem}\label{t4}
Pour tout $n \ge 133$ on a
\begin{equation}\label{t51}
\Pplus(g(n)) \ge 
\sqrt{n \log n}\left(1+\frac{\log_2 n -2.27}{2\log n} \right)
\end{equation}
et, pour tout $n \ge 1755$ on a
\begin{equation}\label{t52}
\Pplus(g(n)) \ge \sqrt{n \log n}.
\end{equation}
\end{theorem}

\begin{proof}
Par l'algorithme naïf exposé dans \cite{DNZ,NICRIRO} on calcule d'abord $g(n)$
et $\Pplus(g(n))$ pour $n \le 10^6$ et l'on s'assure que l'inégalité
\eqref{t51} est vraie pour $133 \le n \le 10^6$, mais fausse pour
$n=132$, puis que \eqref{t52} est vraie pour $1755 \le n \le 10^6$ et
fausse pour $n=1754$.

Notons que pour $n \ge 10^6$ la minoration \eqref{t52}
résulte de $\eqref{t51}$, puisque $\log \log 10^6 \ge 2.27$.
Il nous reste donc à prouver que \eqref{t51} est vraie
pour $n > 10^6$.

Supposons d'abord $n_0 = 10^6 < n \le n_1 = 7.9 \times 10^{21}$,
de sorte que, par \eqref{th1.1}, on ait
\[
P = \Pplus(g(n)) \le 1.266 \sqrt{n_1 \log n_1} \le 8 \times 10^{11}
\]
d'où l'on déduit $\theta(P) < P$ avec \eqref{T1}. 

En utilisant la minoration \eqref{L4} de $\log g(n)$,
il vient par \eqref{m21} et \eqref{m31} 
\[
\sqrt{n \log n} \left(1+\frac{\log_2 n-1.18}{2\log n}\right)
\le \log g(n) < P + 2.13 \sqrt P
\]
et, par la majoration \eqref{th1.1},
\[
P > \sqrt{n \log n} 
\left(1+\frac{\log_2 n -1.18}{2\log n}-\frac{2.4}{(n\log n)^{1/4}} 
\right)\cdot
\]
Mais la fonction $t \mapsto (\log t)^{3/4} t^{-1/4}$ est décroissante
pour $t \ge n_0 = 10^6$, ce qui implique
\[
\frac{2.4}{(n\log n)^{1/4}} 
= \frac{4.8}{2 \log n} \frac{(\log n)^{3/4}}{n^{1/4}}
\le \frac{4.8}{2 \log n} \frac{(\log n_0)^{3/4}}{n_0^{1/4}}
\le \frac{1.09}{2\log n}
\]
et prouve \eqref{t51} lorsque $n_0 < n \le n_1$. 

Supposons maintenant $ n > n_1 = 7.9\ 10^{21}$.
Par \eqref{m41}, $P = \Pplus(g(n))$ vérifie
\[
P > 6.3 \times 10^{11}
\]
et, par \eqref{T4},
\[
\theta(P) - P \le 0.2 \frac{P}{\log^2P} 
< \frac{0.2}{\log(6.3\times 10^{11})}\frac{P}{\log P}
< 0.008 \frac{P}{\log P}\cdot
\]
Par \eqref{m21}, \eqref{m31} et la minoration \eqref{L4},
il vient
\begin{equation}\label{t53}
\sqrt{n \log n}\left(1+\frac{\log_2 n-1.18}{2\log n}\right)
< P + 0.008\frac{P}{\log P} +2.13 \sqrt P.
\end{equation}
Mais, par la majoration \eqref{th1.1},
on a
\[
\frac{P}{\log P} < 
\frac{1.266\sqrt{n \log n}}{\frac{1}{2}\log n +\frac{1}{2}\log_2 n +
\log 1.266}
<
\frac{1.266\sqrt{n\log n}}{\frac{1}{2} \log n}
= 5.064\frac{\sqrt{n\log n}}{2\log n}
\]
et
\begin{multline*}
\sqrt P < 
\sqrt{1.266} (n\log n)^{1/4} 
= \frac{\sqrt{1.266} \sqrt{n\log n}}{2\log n}%
 \left(\frac{2(\log n)^{3/4}}{n^{1/4}} \right)\\
< \frac{\sqrt{n \log n}}{2\log n}%
\left(2\sqrt{1.266}\frac{(\log n_1)^{3/4}}{n_1^{1/4}} \right)
< 0.000 15 \frac{\sqrt{n \log n}}{2\log n}\cdot
\end{multline*}
En observant que
$0.008 \times 5.064 + 2.13 \times 0.00015 < 0.05$,
on déduit de \eqref{t53} que, pour $n \ge 7.9 \times 10^{21}$,
on a
\[
P = \Pplus(g(n)) >
\sqrt{n \log n} \left(1+\frac{\log_2 n -1.23}{2\log n} \right),
\]
ce qui termine la preuve du théorème \ref{t4}.
\end{proof}

\begin{theorem}\label{t5}
\begin{enumerate}
\item
Il existe deux constantes positives $C$ et $a$ telles que,
pour tout $n \ge 2$,
\begin{equation}\label{t61}
\Pplus(g(n)) \ge \sqrt{\Li^{-1}(n)} -C \sqrt{n} \exp(-a \sqrt{\log n}).
\end{equation}
\item
Si la borne supérieure $\Theta$ des parties réelles des zéros
de la fonction $\zeta$ de Riemann est plus petite que $1$,
il existe une constante $C'$ telle que
\begin{equation}\label{t62}
\Pplus(g(n)) \ge \sqrt{\Li^{-1}(n)} - C'n^{\Theta/2}(\log n)^{2+\Theta/2}.
\end{equation}
\end{enumerate}
\end{theorem}

\begin{proof}
\begin{enumerate}
\item
Par \eqref{m21} et \eqref{m31}, $P = \Pplus(g(n))$ satisfait
\begin{equation}\label{t63}
\log g(n) \le \theta(P) + 2.13 \sqrt P.
\end{equation}
Mais, par \eqref{asy}, il existe deux constantes
positives $C_1$ et $a_1$ telles que, pour $n \ge 2$,
\begin{equation}\label{t64}
\log g(n) \ge \sqrt{Li^{-1}(n)} - C_1 \sqrt{n} \exp(-a_1 \sqrt{\log n}).
\end{equation}
Par \eqref{TNP1}, il existe deux constantes
positives $C_2$ et $a_2$ telles que pour tout $x > 1$ on ait
\begin{equation}\label{t65}
\theta(x) \le x + C_2 x \exp(-a_2 \sqrt{\log x}),
\end{equation}
ce qui, par la majoration \eqref{th1.1} entraine:
\begin{equation}\label{t66}
\theta(P) + 2.13 \sqrt P \le P + C_3 \sqrt n \exp(-a_3 \sqrt{\log n})
\end{equation}
avec $C_3 > 0$ et $a_3 > 0$.

Les inégalités \eqref{t63}, \eqref{t64} et \eqref{t66} 
donnent ensuite
\[
P \ge \sqrt{\Li^{-1}(n)} - C_1 \sqrt n \exp(-a_1 \sqrt{\log n})
 - C_3 \sqrt n \exp(-a_3 \sqrt{\log n})
\]
ce qui prouve \eqref{t61} avec $a = \min(a_1,a_3)$ et $C = C_1 + C_3$.
\item
La preuve est similaire en remplaçant les inégalités 
\eqref{t64}, \eqref{t65} et \eqref{t66} par
\begin{eqnarray*}
\log g(n) &\ge& \sqrt{\Li^{-1}(n)} - C'_1(n \log n)^{\Theta/2},
\qquad (\rm{cf.}\ \eqref{asyTH})
\\
\theta(x) &\le& x + C'_2  x^{\Theta} \log^2 x
\qquad (\rm{cf.}\ \eqref{TNP2})
\end{eqnarray*}
et
\[
\theta(P) + 2.13 \sqrt P \le P + C'_3 (n \log n)^{\Theta/2} \log^2 n.
\]
\end{enumerate}
\end{proof}

\section{Comparaison de $\Pplus(g(n)$ et de $\log g(n)$}\label{pplusvslog}

Soit $\cn(x)$ le nombre de $n$, $2 \le n \le x$, tels que
$\Pplus(g(n)) > \log g(n)$. La table des valeurs de $\cn(x)$
\[
\begin{array}{|r|r|r|r|r|}
x =  & 10^3 & 10^4 & 10^5 & 10^6 \\
\hline
\cn(x) = &972 & 9\,787 & 99\,424  & 996\,727\\
\hline
\end{array}
\]
montre que, pour $n \le 10^6$, $\Pplus(g(n))$ a une forte
tendance à excéder $\log g(n)$. Cependant nous allons prouver

\begin{theorem}\label{t6}
Il existe une infinité de valeurs de $n$ telles que
$\Pplus(g(n)) > \log g(n)$ et une infinité de valeurs
de $n$ telles que $\Pplus(g(n)) < \log g(n)$.
\end{theorem}

La démonstration du théorème \ref{t6} repose sur la notion de nombre
$\ell$--superchampion et sur les oscillations des fonctions $\theta$ et
$\psi$ de Chebyshev (cf. le lemme \ref{m5lem} ci-dessous). A partir de
la fonction additive $\ell$ (définie en \eqref{ldef}), les nombres
$\ell$--superchampions sont construits à l'image des nombres hautement
composés supérieurs, introduits par Ramanujan à partir de la fonction
\og nombre de diviseurs\fg. Ces nombres
$\ell$--superchampions ont étés utilisés dans les articles
\cite{NICAA,NICTH,NICCRAS70,MAS,MNRAA,MNRMC,NIC97}.

\begin{definition}
Soit $P \ge 5$ un nombre premier et
\begin{equation}\label{t71}
\rho = \frac{P}{\log P}\cdot
\end{equation}
Pour tout nombre premier $p$ on définit
\begin{equation}\label{t72}
\al_p = \left\{
\begin{array}{ccl}
0 &\text{si}& p > P\\[1ex]
1 &\text{si}& p \le P \text{ et }\rho < \dfrac{p^2-p}{\log p}\\[3ex]
i \ge 2 &\text{si}& \dfrac{p^i-p^{i-1}}{\log p} \le \rho < %
\dfrac{p^{i+1}-p^i}{\log p}\cdot
\end{array}
\right.
\end{equation}
Le nombre $\ell$--superchampion $N_P$ est défini par
\begin{equation}\label{t73}
N_P = \prod_{p} p^{\al_p} = \prod_{p \le P} p^{\al_p}
\end{equation}
et l'on a
\begin{equation}\label{t74}
n_P = \ell(N_P) = \sum_{p \le P} p^{\al_p}.
\end{equation}
\end{definition}
On notera que $\al_p$ est défini de fa\c con à minimiser
la suite $\beta \mapsto \ell(p^{\bt}) - \rho\bt \log p$. On a donc
pour tout $\bt$ entier, $\bt \ge 0$,
\[
\ell(p^{\bt}) - \rho\bt \log p  \ge \ell(p^{\al_p}) - \rho\al_p \log p
\]
ce qui, par l'additivité des fonctions $\ell$ et $\log $ entraine
que, pour tout $M = \prod_{p} p^{\bt_p}$, on a
\begin{equation}\label{t75}
\ell(M) - \rho\log M \ge \ell(N_P)-\rho\log N_P.
\end{equation}

\begin{prop}\label{propt76}
Pour tout $P$ premier, $P \ge 5$, on a
\begin{equation}\label{t77}
g(n_P) = N_P
\end{equation}
et
\begin{equation}\label{t78}
\theta(P) \le \log N_P = \log g(n_P) \le \psi(P)
\end{equation}
où $\theta$ et $\psi$ sont les fonctions de Chebyshev
définies en \eqref{defthetapsi}.
\end{prop}

\begin{proof}
Soit $M$ un entier tel que $\ell(M) \le \ell(N_P) = n_P.$
Par \eqref{t75}, il vient
\[
\rho \log\frac{M}{N_P} \le \ell(M) - \ell(N_P) \le 0
\]
d'o\`u $M \le N_P$ ce qui entraine \eqref{t77} par définition de $g$.
La minoration dans \eqref{t78} résulte de \eqref{t72} et \eqref{t73}.
Pour prouver la majoration, il faut montrer que, pour $i \ge 2$,
\begin{equation}\label{lab86}
\frac{p^i-p^{i-1}}{\log p} \le \rho = \frac{P}{\log P}
\implies p^i \le P.
\end{equation}
Comme la fonction $f(t) = \dfrac{t^i-t^{i-1}}{\log t}$ est
croissante pour tout $i \ge 2$, pour prouver \eqref{lab86}, il suffit
de montrer
\[
f(P^{1/i}) > \frac{P}{\log P}\cdot
\]
On a 
\[
\frac{f(P^{1/i})}{P/\log P} = i\left(1-\frac{1}{P^{1/i}}\right)
\ge i\left(1-\frac{1}{5^{1/i}} \right)\cdot
\]
Mais $2(1-\frac{1}{\sqrt 5}) > 1$ et pour $i \ge 3$,
on a \dsm{5^{1/i} =  \exp\frac{\log 5}{i} > 1 + \frac{\log 5}{i}}
et
\[
 i\left(1-\frac{1}{5^{1/i}} \right) >
 i\left(1-\frac{1}{1+(\log 5)/i} \right) =
\frac{i \log 5}{i+\log 5} \ge \frac{3 \log 5}{3+\log 5} > 1,
\]
ce qui achève la démonstration de \eqref{t78}.
\end{proof}
Nous aurons aussi besoin du lemme ci-dessous, conséquence du
théorème de Littlewood sur les oscillations des fonctions
$\theta$ et $\psi$.

\begin{lem}\label{m5lem}
\mbox{}
\begin{itemize}
\item
Il existe une infinité de nombres premiers $P$ tels que
\begin{equation}\label{m51}
\theta(P) \le \psi(P) < P.
\end{equation}
\item
Il existe une infinité de nombres premiers $P$ tels que
\begin{equation}\label{m52}
P < \theta(P)  \le \psi(P).
\end{equation}
\end{itemize}
\end{lem}

\begin{proof}
Par le théorème de Littlewood 
(cf. \cite{LIT}, \cite{INGHAM32} Th. 34,
\cite{EMF} Th. 6.3) il existe une suite de
valeurs de $x$ tendant vers l'infini et une constante
$c > 0$ telles que
\begin{equation}\label{m53}
\theta(x) \le \psi(x) \le x - c\sqrt x \log_3 x
\end{equation}
et il existe une suite de
valeurs de $x'$ tendant vers l'infini et une constante
$c' > 0$ telles que
\begin{equation}\label{m54}
x' + c'\sqrt x' \log_3 x' \le \theta(x') \le \psi(x').
\end{equation}
Supposons d'abord $x$ assez grand vérifiant \eqref{m53}.
Si $x$ est un nombre premier, \eqref{m53} implique
\eqref{m51}; sinon soit $P$ le nombre premier suivant 
$x$. On a, par \eqref{m53} et la croissance de
$t \mapsto t - c \sqrt t \log_3 t$,
\begin{multline*}
\theta(P) = \\
\theta(x) + \log P \le x - c \sqrt x \log_3 x + \log P
<  P - c \sqrt P \log_3 P + \log P
\end{multline*}
et, par \eqref{T5},
\[
\psi(P) \le \theta(P) + 2 \sqrt P <  P - c \sqrt P \log_3 P +
+ 2 \sqrt P + \log P
\]
qui implique \eqref{m51} lorsque $x$ est assez grand.
\end{proof}

La preuve de \eqref{m52} est plus facile~: soit $x'$ vérifiant
\eqref{m54} et $P$ le plus grand nombre premier satisfaisant
$P \le x'$. On a
\[
\theta(P) = \theta(x') \ge x' + c' \sqrt{x'} \log_3 x' > P.
\qquad \Box
\]

\paragraph{Démonstration du théorème \ref{t6}}
On considère la suite des nombres $n_P$ définis par
\eqref{t74}. Par \eqref{t77}, on a 
$g(n_P) = N_P$ et, par \eqref{t73}, on a $\Pplus(g(n_P)) = P$.
On applique alors \eqref{t78} et le lemme \ref{m5lem}.
Notons que la démonstration prouve en fait
\begin{equation}\label{t79}
\Pplus(g(n)) - \log g(n) =
\Omega_{\pm}\left(n^{1/4} (\log n)^{1/4}\log_3 n \right).
\qquad
\Box
\end{equation}

\renewcommand{\arraystretch}{0.8}
\begin{table}[p]
\caption{ Les valeurs de $\thmin(p)$ pour les premiers
  $\thmin$--champions $p$, arrondies par défaut.  Soit $p < q$ deux
  champions consécutifs. La fonction $\thmin$ est constante sur
  l'intervalle $\intfg{p}{q}$ de valeur $\thmin(p) =
  \dfrac{\theta(^{\star} q)}{q}$ en notant $^{\star} q$ le nombre
  premier précédant $q$.
Pour $x \ge p$, on a $\dfrac{\theta(x)}{x} \ge \,\thmin(p).$}
\label{tablethmin}
\[
\begin{array}[t]{||rr|}
\hline
p & \thmin(p)\\
\hline
2      & 0.2310\\
3      & 0.3583\\
5      & 0.4858\\
7      & 0.4861\\
11     & 0.5957\\
13     & 0.6064\\
17     & 0.6628\\
29     & 0.7033\\
37     & 0.7228\\
41     & 0.7615\\
59     & 0.7928\\
67     & 0.8073\\
71     & 0.8160\\
97     & 0.8289\\
101    & 0.8430\\
\hline
\end{array}
\begin{array}[t]{|rr|}
\hline
p & \thmin(p)\\
\hline
127    & 0.8499\\
149    & 0.8694\\
163    & 0.8695\\
223    & 0.8780\\
227    & 0.8940\\
229    & 0.8980\\
347    & 0.9096\\
349    & 0.9130\\
367    & 0.9134\\
419    & 0.9160\\
431    & 0.9194\\
557    & 0.9208\\
563    & 0.9222\\
569    & 0.9264\\
587    & 0.9278\\
\hline
\end{array}
\begin{array}[t]{|rr|}
\hline
p & \thmin(p)|\\
\hline
593    & 0.9291\\
599    & 0.9367\\
601    & 0.9380\\
607    & 0.9383\\
809    & 0.9409\\
821    & 0.9449\\
853    & 0.9455\\
1423   & 0.9480\\
1427   & 0.9517\\
1429   & 0.9541\\
1433   & 0.9550\\
1447   & 0.9573\\
1451   & 0.9576\\
1481   & 0.9600\\
1973   & 0.9609\\
\hline
\end{array}
\begin{array}[t]{|rr||}
\hline
p & \thmin(p)\\
\hline
1987   & 0.9618\\
1993   & 0.9629\\
2237   & 0.9632\\
2657   & 0.9654\\
2659   & 0.9669\\
3299   & 0.9688\\
3301   & 0.9695\\
3307   & 0.9696\\
3449   & 0.9697\\
3457   & 0.9709\\
3461   & 0.9710\\
3511   & 0.9713\\
3527   & 0.9730\\
3529   & 0.9742\\
3533   & 0.9747\\
\hline
\end{array}
\]
\end{table}

\renewcommand{\arraystretch}{0.7}
\begin{table}[p]
\caption{Les premiers $\theta_d$--champions $p$ 
et les valeurs $\theta_d(p)$, arrondies par excès.
Si $p$ et $q$ sont deux champions consécutifs, on a 
$\theta_d(p)=\left( 1- \dfrac{\theta(^{\star} q)}{q}\right) \log ^2 \,q$ 
en notant $^{\star} q$ le nombre premier précédant $q$.
Pour $p \le x < q$ on a $\theta_d(x)=\theta_d(p)$ et, pour $x \ne 1$,
$1-\dfrac{\theta_d(p)}{\log^2(x)} \le \dfrac{\theta(x)}{x} \le 1.$
}
\label{tableThetad}
\[
\begin{array}{||r|r|}   
\hline
p & \theta_d(p)\\       	
\hline
1&3.964809\\
59&3.850387\\
97&3.813284\\
223&3.588327\\
227&3.488612\\
347&3.220904\\
557&3.174406\\
563&3.127884\\
569&2.989177\\
587&2.942596\\
593&2.896013\\
599&2.868691\\
1423&2.742321\\
1427&2.545152\\
1429&2.419671\\
1433&2.382941\\
1447&2.297149\\
\hline
\end{array}
\begin{array}{|r|r|}
\hline
p & \theta_d(p)\\       	
\hline
1973&2.287253\\
2657&2.168328\\
3299&2.046009\\
3301&2.015399\\
3449&2.009078\\
3457&1.927729\\
3461&1.927471\\
3511&1.914641\\
3527&1.862808\\
5381&1.825717\\
5387&1.788674\\
5393&1.751679\\
5399&1.741574\\
5407&1.704658\\
5413&1.680383\\
7451&1.667866\\
7477&1.614849\\
\hline
\end{array}
\begin{array}{|r|r|}
\hline
p & \theta_d(p)\\       	
\hline
7481&1.582806\\
7487&1.556462\\
11777&1.501244\\
11779&1.460762\\
11783&1.453427\\
11801&1.427708\\
11807&1.402006\\
11813&1.391140\\
11897&1.372839\\
11923&1.362143\\
14387&1.356310\\
19373&1.336505\\
19379&1.296811\\
19381&1.277369\\
19417&1.247671\\
19421&1.208042\\
19423&1.178361\\
\hline
\end{array}
\begin{array}{|r|r||}
\hline
p & \theta_d(p)\\       	
\hline
19427&1.152039\\
19681&1.132553\\
19697&1.114653\\
19913&1.109461\\
20873&1.093478\\
31957&1.093040\\
32051&1.078184\\
32057&1.050007\\
32059&1.041410\\
32297&1.013380\\
32299&0.991977\\
32303&0.982304\\
32321&0.954291\\
32323&0.949739\\
32353&0.934985\\
32359&0.913616\\
32363&0.898870\\
\hline
\end{array}
\]
\end{table}

\renewcommand{\arraystretch}{0.8}
\begin{table}[p]
\caption{Les premiers champions de $\eta_1$.
Si $p$ et $q$ sont deux champions consécutifs
et $p \le x < q$, on a $\eta_1(x) = \eta_1(p)$.
Si $p$ est un champion, pour $x  \ge  p$ 
on a $\eta_1(x) \ge \eta_1(p)$ et
$\intfd{\eta_1(p) x}{x}$
contient au moins un nombre premier.}
\label{tableeta1}
\[
\begin{array}[t]{||r|r||} 
\hline
p & \eta_{1}(p)\\
\hline
2 & 3/5 \\
5 & 7/11 \\
11 & 13/17 \\
17 & 23/29 \\
29 & 31/37 \\
37 & 47/53 \\
53 & 113/127 \\
127 & 139/149 \\
149 & 199/211 \\
211 & 211/223 \\
223 & 293/307 \\
307 & 317/331 \\
331 & 523/541 \\
541 & 1327/1361 \\
1361 & 1669/1693 \\
1693 & 1951/1973 \\
1973 & 2179/2203 \\
\hline
\end{array}
\begin{array}[t]{|r|r||} 
\hline
p & \eta_{1}(p)\\
\hline
2203 & 2477/2503 \\
2503 & 2971/2999 \\
2999 & 3271/3299 \\
3299 & 4297/4327 \\
4327 & 4831/4861 \\
4861 & 5591/5623 \\
5623 & 5749/5779 \\
5779 & 5953/5981 \\
5981 & 6491/6521 \\
6521 & 6917/6947 \\
6947 & 7253/7283 \\
7283 & 8467/8501 \\
8501 & 9551/9587 \\
9587 & 9973/10007 \\
10007 & 10799/10831 \\
10831 & 11743/11777 \\
11777 & 15683/15727 \\
\hline
\end{array}
\begin{array}[t]{|r|r||} 
\hline
p & \eta_{1}(p)\\
\hline
15727 & 19609/19661 \\
19661 & 31397/31469 \\
31469 & 34061/34123 \\
34123 & 35617/35671 \\
35671 & 35677/35729 \\
35729 & 43331/43391 \\
43391 & 44293/44351 \\
44351 & 45893/45943 \\
45943 & 48679/48731 \\
48731 & 58831/58889 \\
58889 & 59281/59333 \\
59333 & 60539/60589 \\
60589 & 79699/79757 \\
79757 & 89689/89753 \\
89753 & 107377/107441 \\
107441& 155921/156007 \\
156007& 188029/188107\\
\hline
\end{array}
\]
\end{table}

\renewcommand{\arraystretch}{0.7}
\begin{table}[p]
\caption{Les premiers champions de $\eta_2$.
Si $p$ et $q$ sont deux champions consécutifs
et $p \le x < q$, on a $\eta_2(x) = \eta_2(p)$.
Si $p$ est un champion, pour $x  \ge  p$ 
on a $\eta_2(x) \ge \eta_2(p)$ et
$\intfd{\eta_2(p) x}{x}$
contient au moins 2 nombres premiers.}
\label{tableeta2}
\[
\begin{array}[t]{||r|r||}  
\hline
p & \eta_{2}(p)\\
\hline        	
3 & 2/5 \\
5 & 3/7 \\
7 & 5/11 \\
11 & 7/13 \\
13 & 11/17 \\
17 & 19/29 \\
29 & 23/31 \\
31 & 31/41 \\
41 & 47/59 \\
59 & 83/97 \\
97 & 109/127 \\
127 & 113/131 \\
131 & 199/223 \\
223 & 283/307 \\
307 & 317/337 \\
337 & 331/347 \\
347 & 523/547 \\
\hline
\end{array}
\begin{array}[t]{|r|r||}          	
\hline
p & \eta_{2}(p)\\
\hline          	
547 & 619/641 \\
641 & 773/797 \\
797 & 1321/1361 \\
1361 & 1327/1367 \\
1367 & 1381/1409 \\
1409 & 2161/2203 \\
2203 & 2477/2521 \\
2521 & 3121/3163 \\
3163 & 3259/3299 \\
3299 & 3947/3989 \\
3989 & 4159/4201 \\
4201 & 4297/4337 \\
4337 & 4817/4861 \\
4861 & 5591/5639 \\
5639 & 5939/5981 \\
5981 & 6481/6521 \\
6521 & 7253/7297 \\
\hline          	
\end{array}
\begin{array}[t]{|r|r||}
\hline
p & \eta_{2}(p)\\
\hline          	
7297 & 7963/8009 \\
8009 & 8467/8513 \\
8513 & 9551/9601 \\
9601 & 9967/10007 \\
10007 & 11003/11047 \\
11047 & 14087/14143 \\
14143 & 19609/19681 \\
19681 & 24251/24317 \\
24317 & 25471/25537 \\
25537 & 31397/31477 \\
31477 & 38461/38543 \\
38543 & 58789/58889 \\
58889 & 58831/58897 \\
58897 & 62233/62297 \\
62297 & 69557/69623 \\
69623 & 74941/75011 \\
75011 & 79699/79769 \\
\hline          	
\end{array}
\]
\end{table}

\renewcommand{\arraystretch}{0.7}
\begin{table}[htpb]
\caption{Les premiers champions de $\eta_3$.
Si $p$ et $q$ sont deux champions consécutifs
et $p \le x < q$, on a $\eta_3(x) = \eta_3(p)$.
Si $p$ est un champion, pour $x  \ge  p$ 
on a $\eta_3(x) \ge \eta_3(p)$ et
$\intfd{\eta_3(p) x}{x}$
contient au moins 3 nombres premiers.}
\label{tableeta3}
\[
\begin{array}{||r|r|}          	
\hline
p & \eta_{3}(p)\\
\hline
5 & 3/11 \\
11 & 5/13 \\
13 & 7/17 \\
17 & 13/23 \\
23 & 17/29 \\
29 & 19/31 \\
31 & 23/37 \\
37 & 29/41 \\
41 & 31/43 \\
43 & 43/59 \\
59 & 47/61 \\
61 & 53/67 \\
67 & 79/97 \\
97 & 83/101 \\
101 & 113/137 \\
137 & 199/227 \\
227 & 283/311 \\
\hline
\end{array}
\begin{array}{|r|r|}
\hline
p & \eta_{3}(p)\\
\hline
311 & 317/347 \\
347 & 523/557 \\
557 & 773/809 \\
809 & 887/919 \\
919 & 1321/1367 \\
1367 & 1327/1373 \\
1373 & 1381/1423 \\
1423 & 2153/2203 \\
2203 & 2477/2531 \\
2531 & 2551/2591 \\
2591 & 3121/3167 \\
3167 & 3947/4001 \\
4001 & 4159/4211 \\
4211 & 4817/4871 \\
4871 & 5581/5639 \\
5639 & 5927/5981 \\
5981 & 5953/6007 \\
\hline
\end{array}
\begin{array}{|r|r||}
\hline
p & \eta_{3}(p)\\
\hline
6007 & 6491/6547 \\
6547 & 7351/7411 \\
7411 & 7759/7817 \\
7817 & 7951/8009 \\
8009 & 9551/9613 \\
9613 & 9973/10037 \\
10037 & 10369/10427 \\
10427 & 11177/11239 \\
11239 & 11719/11777 \\
11777 & 12829/12889 \\
12889 & 13933/13997 \\
13997 & 14087/14149 \\
14149 & 14563/14621 \\
14621 & 19603/19681 \\
19681 & 19609/19687 \\
19687 & 24251/24329 \\
24329 & 35617/35729 \\
\hline
\end{array}
\]
\end{table}
 
\begin{table}[p]
\caption{Les premiers $\dt_3$--champions $p$
et les valeurs $\dt_3(p)$, arrondies par excès~:
Soit $p$ et $q$  deux champions consécutifs.
On a $\dt_3(p)=(1-\eta_3(^\star q))\log^2 q$
en notant $^\star q$ le nombre premier précédant $q$.
Pour $p \le x < q$ on a $\dt_3(x) = \dt_3(p)$ et, pour $x\ge p$,
$1-\dfrac{\dt_3(p)}{\log^2x} \le \eta_3(x) \le 1.$}\label{tabledelta3}
\[
\begin{array}[t]{||r|r||}  
\hline
p & \dt_3(p)\\[0.8ex]
\hline        	
5  & 4.9336\\
37  & 4.5089\\
59 & 4.2406\\
137 & 3.6302\\
227 & 2.9662\\
311 & 2.9581\\
347 & 2.4402\\
557 & 1.9951\\
809& 1.7544\\
1367& 1.7488\\
1373& 1.5559\\
1423& 1.3449\\
2203& 1.3102\\
\hline
\end{array}
\begin{array}[t]{|r|r|}  
\hline
p & \dt_3(p)\\[0.8ex]
\hline        	
 2531& 0.9538\\
 2591& 0.9438\\
 3167& 0.9286\\
 4001& 0.8601\\
 4211& 0.7993\\
 4871& 0.7674\\
 5639& 0.6828\\
 5981& 0.6806\\
 6007& 0.6604\\
 6547& 0.6429\\
 7411& 0.5963\\
 7817& 0.5851\\
 8009& 0.5425\\
\hline
\end{array}
\begin{array}[t]{||r|r||}  
\hline
p & \dt_3(p)\\[0.8ex]
\hline        	
 9613& 0.5414\\
 10037& 0.4800\\
 11239& 0.4328\\
 11777& 0.4170\\
 12889& 0.4168\\
 13997& 0.4003\\
 14149& 0.3875\\
 19681& 0.3874\\
 19687& 0.3446\\
 35729& 0.2777\\
 38557& 0.2292\\
 43889& 0.2213\\
 58897 & 0.1714\\
\hline
\end{array}
\]
\end{table}

\pagebreak

\bibliographystyle{plain}

\medskip

\hspace{-2cm}
\begin{minipage}{14cm}
Marc Deléglise, Jean-Louis Nicolas\\
Université de Lyon, Université Lyon 1, CNRS\\
Institut Camille Jordan, Mathématiques,
Bât. Doyen Jean Braconnier\\
Université Claude Bernard (Lyon 1)\\
21 avenue Claude Bernard
F-69622 Villeurbanne cédex, France.\par

e-mail~:  \url{deleglise@math.univ-lyon1.fr},\quad \url{jlnicola@in2p3.fr}

\url{http://math.univ-lyon1.fr/~deleglis/}\\
\url{http://math.univ-lyon1.fr/~nicolas/}
\end{minipage}
\end{document}